\documentclass[11pt,reqno]{amsart}
\usepackage{amsmath,amssymb}
\usepackage[all]{xy}

\usepackage{hyperref}
\hypersetup{
	colorlinks,
	citecolor=red,
	filecolor=black,
	linkcolor=blue,
	urlcolor=black
}

\usepackage{pdfpages}
\usepackage[labelformat=empty]{caption}
\usepackage{graphicx}
\usepackage{verbatim}
\usepackage{inputenc}

\newtheorem{thm}{Theorem}[section]

\newtheorem{lmm}[thm]{Lemma}

\theoremstyle{definition}
\newtheorem{rem}[thm]{Remark}
\newtheorem{CC}[thm]{Consistency checks}

\numberwithin{equation}{section}

\providecommand \xii{\mathcal{L}}
\providecommand \yii{\mathbb{L}}

\providecommand \B{\mathrm{B}}

\def \hf{\hspace*{0.5cm}}

\begin{document}

\title[Genus two enumerative invariants in del-Pezzo surfaces]{Genus two
enumerative invariants in del-Pezzo surfaces with a fixed complex structure}

\author[I. Biswas]{Indranil Biswas}

\address{School of Mathematics,
Tata Institute of fundamental research, Homi Bhabha road, Mumbai 400005, India}

\email{indranil@math.tifr.res.in}

\author[R. Mukherjee]{Ritwik Mukherjee}

\address{School of Mathematics, National Institute of Science Education and Research, Jatni, Odisha 752050, India}

\email{ritwikm@niser.ac.in}

\author[V. Thakre]{Varun Thakre}

\address{International Centre for Theoretical Sciences, Bengaluru, Karnataka 560089, India}

\email{varun.thakre@icts.res.in}



\date{}


\begin{abstract}
We obtain a formula for the number of genus two curves with a fixed 
complex structure of a given degree on a del-Pezzo surface that pass 
through an appropriate number of generic points of the surface.
This is done by extending the symplectic approach of Aleksey Zinger.
This enumerative problem is expressed as the difference between the 
symplectic invariant and an intersection number on the moduli space of 
rational curves on the surface.
\end{abstract}

\maketitle

\section{Introduction}\label{introduction}

Enumerative Geometry of rational curves in $\mathbb{P}^2_{\mathbb C}$ is a classical 
question in algebraic geometry. A natural generalization of it is to ask how many 
genus $g$ curves with a fixed complex structure
are there of a given degree that pass through 
the right number of generic points. In \cite{Rahul_genus_1} and \cite{ionel_genus1} 
using methods of algebraic and symplectic geometry respectively, Pandharipande and 
Ionel obtain the following result: 

\begin{thm}[{R.Pandharipande and E.Ionel}]\label{qu}
There is an explicit formula for the number of 
degree $d$ genus one curves with a fixed complex structure 
in $\mathbb{P}^2_{\mathbb{C}}$ that pass through 
$3d-1$ generic points. 
\end{thm}

In \cite{KQR} and \cite{g2p2and3} 
by extending the method of Pandharipande and Ionel respectively, 
Katz, Qin and Ruan and 
Zinger obtain the following results: 

\begin{thm}[{S.Katz, Z.Qin, Y.Ruan and A. Zinger}]\label{qu_genus_two}
There is an explicit formula for the number of 
degree $d$ genus two curves with a fixed complex structure 
in $\mathbb{P}^2_{\mathbb{C}}$ that pass through 
$3d-2$ generic points.
\end{thm}

In \cite{IBRMVT}, we extended Theorem \ref{qu} to del-Pezzo surface. 
Our aim here is to extend Theorem \ref{qu_genus_two} for del-Pezzo surfaces. 

Let $X$ be a complex del-Pezzo surface and $\beta \,\in\, H_2(X,\, \mathbb{Z})$
a given homology class. Let $n_{g, \beta}^j$ denote the number of 
genus $g$ curves with a fixed complex structure that pass through the 
right number of generic points. For notational convenience, the number
$n_{0, \beta}^j$ will be denoted by $n_{0, \beta}$. Given a (co)homology 
class $\alpha$ in $X$, we will denote its Poincar\'{e} dual in $X$ 
by $\hat{\alpha}$. We prove the following:

\begin{thm}\label{main_thm}
Let $X$ 
be a complex del-Pezzo surface and $\beta \,\in\, H_2(X,\, \mathbb{Z})$
a given homology class. 
Denote 
\[ x_i := c_i(TX), \qquad \textnormal{and} \qquad \delta_{\beta} := \langle x_1, ~\beta \rangle -1, \] 
where $c_i$ denotes the $i^{\textnormal{th}}$ Chern class. 
Let
$n_{2, \beta}^{j}$ denote 
the number of 
genus two curves with fixed complex structure $j$ of degree $\beta$
in $X$ that pass through $\delta_{\beta} -1$ generic points. 
Let us make the following assumptions on $\beta$. 
If $X$ is $\mathbb{P}^2$ blown up at $k\leq 8$ points, then suppose  
\[ \beta := d L - m_1 E_1 -\ldots -m_k E_k\] 
and assume that 
$n_{0, \beta -3L} \,>\,0$ and $d >2$ 
(here $L$ denotes the homology class of a line and $E_i$ denotes the exceptional divisors).
If $X \,:=\, \mathbb{P}^1 \times \mathbb{P}^1$, then 
suppose 
\[ \beta := a e_1 + b e_2 \] 
and assume that $a$ and $b$ are both greater than $2$ 
(here $e_1$ and $e_2$ denote the classes of 
$[\mathbb{P}^1] \times [\textnormal{pt}]$ and $[\textnormal{pt}]\times [\mathbb{P}^1]$ in $\mathbb{P}^1 \times \mathbb{P}^1$). 
Then,
\begin{align}
\frac{|\textnormal{Aut}(\Sigma_2^j)|}{2}n^{j}_{2, \beta} &=
n_{0, \beta} \Big((2 + b_2(X)) \beta^2 - 10 x_2([X]) - \hat{x}_1^2 + 
\frac{12 \hat{x}_1^2}{\beta \cdot \hat{x}_1}\Big) \nonumber \\ 
& + \sum_{\substack{\beta_1+ \beta_2= \beta, \\ \beta_1, \beta_2 \neq 0}} \binom{\delta_{\beta}-1}{\delta_{\beta_1}}
n_{0, \beta_1} n_{0, \beta_2} (\beta_1 \cdot \beta_2) \Big(
-\frac{6(\beta_1 \cdot \hat{x}_1) (\beta_2 \cdot \hat{x}_1)}{(\beta \cdot \hat{x}_1)} +\frac{\beta_1^2 \beta_2^2 }{2} + 10 \Big) 
\label{main_formula_blwo_up}
\end{align}
where 
$|\textnormal{Aut}(\Sigma_2^j) |$ is the order of the group of holomorphic automorphisms 
of a genus two Riemann surface with fixed complex structure $j$, $b_2(X)$ denotes the 
second Betti number of $X$ 
and ``$\cdot$'' denotes topological intersection. 
\end{thm}

\begin{rem}
The number $\delta_{\beta}-1 := \langle x_1, \beta \rangle -2$ is the index of the linearization of 
$\overline{\partial}$ operator; this is explained in \cite[Pages 41 and 42]{McSa}. 
This is the ``expected'' dimension of the moduli space of genus $2$, degree $\beta$ curves. 
Hence, if the 
complex structure on the target space is regular (i.e. the linearization of the $\overline{\partial}$ operator is surjective) 
then the actual dimension of the moduli space is same as the expected dimension. Since passing through each generic point cuts the 
dimension of the moduli space by one, we conclude that $\delta_{\beta}-1$ are the ``right'' number of generic points to get a 
finite number. In section \ref{genus_two_regular} we show that the complex structure on del-Pezzo surfaces are genus two regular 
(provided $\beta$ satisfies the restrictions we have imposed).
\end{rem}

\begin{rem}
\label{rem_consistency_check}
We expect that \eqref{main_formula_blwo_up} is valid even when 
the conditions we imposed on $\beta$ are not satisfied. 
When $X$ is $\mathbb{P}^2$ blown up at a certain number of points, 
we imposed the condition $n_{0, \beta-3L} >0$ so that we can use 
our formula in \cite{IB_SDM_RM_VP} 
for the number of rational cuspidal curves on del-Pezzo surfaces 
through $\delta_{\beta}-1$ generic points. However, as we explain in the
introduction of \cite{IB_SDM_RM_VP}, we expect that formula  
to hold even when $n_{0, \beta-3L} =0$. The condition $n_{0, \beta-3L} > 0$ was imposed to prove a 
transversality result (\cite[Lemma 6.2, Page 14]{IB_SDM_RM_VP}); that condition is a sufficient condition to prove transversality, it may not be 
necessary. The condition on $d >2$ is imposed to ensure that the complex structure on the del-Pezzo surface 
is genus two regular (as proved in section \ref{genus_two_regular}). \\ 
\hf \hf When $X$ is $\mathbb{P}^1 \times \mathbb{P}^1$, we use the formula 
obtained by Kock \cite{JKoch2} for the number of rational cuspidal curves on $\mathbb{P}^1 \times \mathbb{P}^1$
through $\delta_{\beta}-1$ generic points. There are no assumptions on $\beta$ required to use Kock's formula. 
However, we need to impose the condition $a$ and $b$ are greater than two to ensure that 
the complex structure is genus two regular (as proved in section \ref{genus_two_regular}). 
However, based on the numerical evidence, we expect that the formula  
\eqref{main_formula_blwo_up} is unconditionally true. The conditions we impose on $\beta$ is a sufficient 
condition to prove the necessary transversality/regularity results; they may not be necessary. 
\end{rem}

In \cite{KoMa} and \cite{Pandh_Gott}
a recursive formula is given to compute the numbers $n_{0, \beta}$ 
for del-Pezzo surfaces. Hence, Theorem \ref{main_thm} 
gives $n_{2, \beta}^j$ explicitly. 
We have written a
$\mbox{{\bf C}++}$ 
program to implement the formula in 
Theorem \ref{main_thm} when $X$ is $\mathbb{P}^2$ blown up at $k \leq 8$ points 
and compute $n_{2, \beta}^j$; the program is available in
the web-page 
\[ \textnormal{\url{https://www.sites.google.com/site/ritwik371/home}} \]

\begin{CC} 
We will now present two consistency checks; the second one does not 
quite apply since the restriction we impose on $\beta$ is not satisfied in the second case. \\
\hf \hf Let $X$ be $\mathbb{P}^2$ blown up at 
$k \leq 8$ points. We claim that 
\[ n_{2, d L + \sigma_1 E_1 + \ldots +\sigma_r E_r}^j\,
=\, n^j_{2, dL + \sigma_1 E_1 + \ldots + \sigma_{r-1}E_{r-1}}, \] 
if each of the $\sigma_i$ is $-1$ or $0$. 
Let us explain why this is so. Let us take $X$ to be $\mathbb{P}^2$ blown up at one point. Let us call that point $p$. 
Let us consider the number $n^j_{2, dL -E_1}$; this is the number of genus two curves in $X$ representing the 
class $dL-E_1$ and passing through $3d-3$ generic points. Let $\mathcal{C}$ be one of the curves 
counted by the above number. The curve $\mathcal{C}$ intersects the exceptional divisor exactly at one point. 
Furthermore, since the $3d-3$ points are generic, they can be chosen not to lie in the 
exceptional divisor; let us call the points $p_1, p_2, \ldots, p_{3d-3}$. 
Hence, when we consider the blow down from $X$ to $\mathbb{P}^2$, the curve $\mathcal{C}$ becomes a curve in $\mathbb{P}^2$ passing through 
$p_1, p_2, \ldots, p_{3d-3}$ and the blow up point $p$. Hence, we get a genus two, degree $d$ curve in $\mathbb{P}^2$ passing through 
$3d-2$ points. Hence, there is a one to one correspondence between curves representing the class $dL-E_1$ in $X$ passing through $3d-3$ points 
and degree $d$ curves in $\mathbb{P}^2$ passing through $3d-2$ points. Hence $n^j_{2, dL-E_1} = n^j_{2,dL}$. A similar argument holds when there are 
more than one blowup points. This argument also shows that $n^j_{2, dL+0E_1} = n^j_{2, dL}$; the same reasoning holds by taking a 
curve in the blowup and then considering its image under the blow down. The blow down gives a one to one correspondence between the two sets and 
hence, the corresponding numbers are the same. \\ 
\hf \hf We have verified this assertion in many cases. For instance 
we have verified that 
$$
n_{2, 4L -E}^j \, =\, n_{2, 4L+0E}^j \,=\, n_{2, 4L}^j\, . 
$$
The reader is invited to use our program and verify these assertions. \\ 
\hf \hf The second consistency check we will present is actually not valid, because 
the restrictions we impose on $\beta$ are not going to be satisfied. However, as 
we explained earlier, the restrictions imposed on $\beta$ are sufficient conditions for our formula 
to hold, they may not be necessary. We expect our formula to hold without any restrictions on $\beta$ 
based on this numerical evidence we present. \\ 
\hf \hf Let $g_{\beta}$ denote the genus of smooth degree $\beta$ curve in $X$. 
Then, 
$$g_{\beta} \,:= \,\frac{\beta^2 - \hat{x}_1 \cdot \beta +2}{2}\, .$$ 
Hence, if $g_{\beta} =0$ 
or $1$ then $n_{2, \beta}^j$ ought to be zero. 
We have verified this assertion for several values of $\beta$ using our program. 
For instance, we have verified that if $X$ is $\mathbb{P}^2$ blown up at 
$k \leq 8$ points and $\beta$ is any one of the following 
homology classes, 
\[ (1,0),\ (2,0),\ (3,0),\ (1,-1),\ (2,-1),\ (3, -1),\ (4, -2, -2) \qquad \textnormal{or} \qquad (4,-2,-2,-2) \] 
then $n_{2, \beta}^j$ is indeed zero (as expected). 
However, the hypothesis $n_{0, \beta-3L} > 0$ doesn't hold in these cases; hence  
this consistency check is not quite applicable. \\
\hf \hf Next, we can directly verify that when $X\,:= \,\mathbb{P}^1 \times \mathbb{P}^1$, 
$n_{2, \beta}^j$ is zero if $\beta$ is of bi-degree $(a, 0)$, $(a,1)$ (no restriction
on $a$) for some $a$ or if $\beta \,:=\, (2,2)$. In order to do that we directly 
substitute all the values in equation \eqref{main_formula_blwo_up} and compute $n_{2, \beta}^j$, 
which gives us zero in all the cases we have mentioned.\\
\hf \hf This is as expected since 
$g_{\beta}$ is negative, zero or one respectively in those cases. 
Hence $n_{2, \beta}^j$ should be zero in those cases. Again, this consistency check is not quite valid, since 
the hypothesis $a>2$ and $b>2$ is not satisfied.
\end{CC}

\begin{rem}
In \cite{ZingKQR}, Zinger corrects an error in \cite{KQR} and 
obtains the same formula as in \cite{g2p2and3}. 
The method presented in this paper is an extension of the symplectic approach
employed in \cite{g2p2and3}; it would be interesting to see if the algebro-geometric 
method presented in \cite{KQR} and \cite{ZingKQR} can be extended to obtain our 
formula \eqref{main_formula_blwo_up}.
\end{rem}

\begin{rem}
\label{rem_zinger_spl_case}
Let us now explain how we can recover the formula of Zinger in Theorem $1.1$ of \cite{g2p2and3}. 
First, let us recall the formula for the number of rational degree $d$ curves in $\mathbb{P}^2$ through 
$3d-1$ points obtained by Ruan-Tian and Kontsevich-Manin (\cite[page 363]{RT} and \cite{KoMa})
\begin{align}
n_{0,d} & = \frac{1}{6(d-1)} \sum_{d_1 + d_2 = d} 
\binom{3d-2}{3d_1-1} d_1 d_2 n_{0, d_1} n_{0, d_2} \Big( d_1 d_2 - 2 \frac{(d_1-d_2)^2}{3d-2} \Big). \label{n_d_rt} 
\end{align}
The above formula is only for $\mathbb{P}^2$ (and not other del-Pezzo surfaces).
Let us now write down equation \eqref{main_formula_blwo_up} 
with $X:=\, \mathbb{P}^2$ and $|\textnormal{Aut}(\Sigma_2^j)| \,=\, |\mathbb{Z}_2| = 2$; that gives us 
\begin{align}
n_{2,d}^j &= 3 (d^2-1) n_{0,d} + n_{0,d} \Big( -36 + \frac{36}{d} \Big) + \sum_{d_1+ d_2 = d} \binom{3d-2}{3d_1-1} 
n_{0,d_1} n_{0, d_2} d_1 d_2 
\Big( -\frac{18 d_1 d_2}{d} + \frac{d_1^2 d_2^2}{2} + 10 \Big). \label{n_2d_zinger_intermediate}
\end{align}
Now substitute the value of $n_{0,d}$ from equation \eqref{n_d_rt} in \eqref{n_2d_zinger_intermediate} 
only in the term $n_{0,d}(-36 + \frac{36}{d})$; keep the first term unchanged. That gives us 
Zinger's formula in Theorem $1.1$ of \cite{g2p2and3}, namely 
\begin{align*}
n_{2,d}^j &= 3 (d^2-1) n_{0,d} + \frac{1}{2} 
\sum_{d_1 + d_2 = d} \binom{3d-2}{3d_1-1} d_1 d_2 n_{0, d_1} n_{0, d_2}
\Big( d_1^2 d_2^2 + 28 -16 \frac{9 d_1 d_2 -1}{3d-2} \Big). 
\end{align*}
The reason we set 
$\textnormal{Aut}(\Sigma_2^j)\,=\, \mathbb{Z}_2$ is because Zinger states his formula for a generic $j$.
However, his arguments go through for any $j$; one simply divides out the difference between 
between the symplectic invariant and the correction term by a different factor (see remark 
\ref{automorphisms}). 
\end{rem}

Recently, the problem of enumerating genus $g$ curves with a fixed complex structure 
has also been studied by tropical geometers. 
In \cite{Kerber_Markwig}, Kerber and Markwig 
compute the number of tropical elliptic curves in $\mathbb{P}^2$ with a fixed $j$-invariant. 
A priori this number need not be the same as the number of elliptic curves in $\mathbb{P}^2$ 
(with a fixed $j$-invariant) that pass through $3d-1$ generic points. In  \cite[Theorem A]{Yoav_Len_Raghunathan}, 
Len and Ranganathan prove that there is a one to one correspondence between 
the Tropical count of curves obtained by Kerber and Markwig in \cite{Kerber_Markwig} 
and the actual number of curves (i.e. the enumerative number). This gives a Tropical geometric 
proof of Theorem \ref{qu}. 
In \cite[Theorem B]{Yoav_Len_Raghunathan}, 
Len and Ranganathan also 
obtain a formula for the number of elliptic 
curves with a fixed $j$-invariant of a given degree 
for Hirzebruch surfaces, using methods from tropical geometry 
(i.e they obtain the formula for genus one enumerative invariants of Hirzebruch surfaces, not just 
the tropical count of curves).\\ 
\hf \hf The problem of counting genus two curves with a fixed complex structure is also 
currently being investigated by tropical geometers. It would be interesting if Zinger's 
formula (and more generally, the formula obtained here) can be recovered by 
tropical methods. 

\section{Enumerative versus symplectic invariant}

Let us now explain the basic idea to compute $n^j_{2, \beta}$.
We will be closely following the discussion and recalling the notions introduced in \cite[Pages 261-263]{RT}. 
Let $(X\, , \omega)$ be a compact semipositive symplectic manifold
of dimension $2m$ and $\beta \,\in\, H_2(X, \mathbb{Z})$ 
a homology class. Let $k$ be a nonnegative integer such that 
$k+ 2 g \,\geq\, 3$. 
Let $[\alpha_1], \cdots, [\alpha_k]$ and $[\gamma_1], \cdots, [\gamma_l]$ 
be integral homology classes in $H_*(X, \mathbb{Z})$ such that 
\begin{align}
\sum_{i=1}^{k} 2 m - \textnormal{deg}(\alpha_i) +
\sum_{j=1}^{l} (2m -2 - \textnormal{deg}(\gamma_j)) \,=\, 
2m(1-g) + 2 \langle c_1(TX), ~\beta \rangle\, . \label{dim_condition}
\end{align}
Fix cycles $A_i$, $1\,\leq\, i\, \leq\, k$, and $B_j$, $1\,\leq\, j\, \leq\, l$,
on $X$ representing the homology classes $\alpha_i$ and $\gamma_j$. Fix a compact
Riemann surface $\Sigma_g$ of genus $g$; its complex structure will be denoted by $j$.
Define 
$$
\mathcal{M}_{g, k}^{\nu, j} (X, \beta; \alpha_1, \cdots, \alpha_k; \gamma_1,
\cdots, \gamma_l)\,:=\, \{ (u, y_1, \cdots, y_k) \,\in\,
\mathcal{C}^{\infty}(\Sigma_g, X) \times \Sigma_g^k\,\mid
$$
$$
\, \overline{\partial}_{j, J} u  \,=\, \nu ,\ 
u(y_i) \,\in\, A_i ~~\forall ~i = 1, \cdots ,k,\ 
~~\textnormal{Im}(u)\cap \B_j\,\not=\,\emptyset ~~\forall j =1, \cdots ,l\}\, ,
$$
where $\nu$ is generic smooth $(0,1)$ form, i.e. 
$\nu \in \Gamma(\Sigma_g \times X, \Lambda^{0,1}T^*\Sigma_g \otimes TX)$. 
Note that the last condition says that 
the image of $u$ intersects $B_i$; it \textit{does not} say that a specific marked point lands on $B_i$. 
The previous condition says that a specific marked point lands on $A_i$. These two conditions are different.\\  
\hf \hf The Symplectic Invariant (or the Ruan--Tian invariant) is defined to be the 
signed cardinality of the above set, i.e., 
$$
\mathrm{RT}_{g, \beta}(\alpha_1, \cdots, \alpha_k; \gamma_1, \cdots, \gamma_l)\,:=\, 
\pm |\mathcal{M}_{g, k}^{\nu, j} (X, \beta; \alpha_1, \cdots, \alpha_k; \gamma_1, \cdots,
\gamma_k)|\, . 
$$
When $k=0$, we denote the invariant as 
$$\mathrm{RT}_{g, \beta}(\emptyset ; \gamma_1, \cdots, \gamma_l).$$ 
Similarly, when $l=0$ we denote the invariant as 
$$\mathrm{RT}_{g, \beta}(\alpha_1, \cdots, \alpha_k; \emptyset).$$
If \eqref{dim_condition} is not satisfied, then 
we formally define the invariant to be zero. 

A natural question to ask is whether the symplectic invariant coincides with the 
enumerative invariant $n_{g, \beta}^j$. 
It is stated in \cite[page 267]{RT} that when $g\,=\,0$ and $X$ is a del-Pezzo surface, the 
Symplectic Invariant is same as the enumerative invariant, meaning,
\[ \mathrm{RT}_{0,\beta}(\emptyset; p_1, \cdots, p_{\delta_{\beta}-1}) \,=\, n_{0, \beta}\, .\] 
However, starting from $g=1$, the enumerative invariant is no longer the same as 
the Symplectic Invariant. Let us give a brief explanation as to why this is the case.\\ 
\hf \hf In general, the following statement is true 
\begin{equation}\label{rt_equal_eg_plus_cr}
\mathrm{RT}_{g,\beta} \,=\, |\textnormal{Aut}(\Sigma_g^j) | n^j_{g, \beta}
+ \mathrm{CR}_{g, \beta}\, ,
\end{equation}
where $\mathrm{CR}_{g, \beta}$ denotes a correction term. Let us explain what this
term means and why it arises. 
First we note that the factor of $|\textnormal{Aut}(\Sigma_g^j)|$ is there because
in the definition of the Symplectic Invariant we do not mod out by the automorphisms of the domain $\Sigma$. 
Hence, if $u\,:\, (\Sigma_g, j)\,
\longrightarrow\, X$ is a solution to the $\overline\partial$--equation, then there will 
be $|\textnormal{Aut}(\Sigma_g^j)|$ new solutions close to $u$ to the perturbed 
$\overline\partial$--equation.\\ 
\hf \hf Next, we note that when $g\,=\,2$, as $\nu
\,\rightarrow\, 0$, a sequence of
$(J\, ,\nu)$-holomorphic maps can also converge to a bubble map. 
These bubble maps will also contribute to the computation of $\mathrm{RT}_{2,\beta}$ invariant. 
This extra contribution is defined to be the correction term $\mathrm{CR}_{2, \beta}$. 
This correction term 
$\mathrm{CR}_{2, \beta}$ is computed in 
\cite{zinger_phd} when $X\,:=\, \mathbb{P}^2$ by expressing 
it as the intersection of certain tautological classes on the moduli space of 
rational curves and in terms of the number of rational cuspidal curves. 
In section \ref{boundary_contribution} we modify 
the arguments presented by Zinger in \cite{zinger_phd} 
and compute the correction term for del-Pezzo surfaces 
(one of the necessary steps there is to compute the characteristic number of rational curves with a cusp; this is 
computed in our paper \cite{IB_SDM_RM_VP} which we use in this paper). In section \ref{SI}, we 
compute the genus two Symplectic Invariant for del-Pezzo surfaces, using 
formulas 
$(1.1)$ and $(1.2)$ in \cite[page 263]{RT}. Combining these two and using \eqref{rt_equal_eg_plus_cr}, we obtain the 
genus two enumerative invariant $n^j_{2,\beta}$. 

\begin{rem}
Let us now explain why the Symplectic and Enumerative invariants of del-Pezzo surfaces are well defined.  
The former is well defined because del-Pezzo surfaces are semi-positive Symplectic manifolds. Any K{\"a}hler surface is a semi positive symplectic manifold.  
To see why this is so, let us recapitulate the definition of a semi-positive Symplectic manifold from \cite[Definition 6.4.1, Page 156]{McSa}: 
a $2n$-dimensional symplectic manifold $(X, \omega)$ is semi-positive if for every spherical homology class $\beta \in  H_2(X, \mathbb{Z})$ 
we have 
\[ \textnormal{if} \qquad \langle \omega, \beta \rangle >0 \qquad \textnormal{and} \qquad \langle c_1(TX), \beta \rangle \geq 3-n \qquad \textnormal{then} \qquad  \langle c_1(TX), \beta \rangle \geq 0.  \] 
This definition clearly holds for any K{\"a}hler surface (since $n =2$). \\
\hf \hf We now note that Symplectic Invariants (as defined in \cite[Pages 260-263]{RT}) are well defined for any semi-positive Symplectic manifold. 
This is summarized in \cite[Theorem 7.2, Page 263]{RT}, where the authors say that the composition laws 
$(1.1)$ and $(1.2)$ in \cite[page 263]{RT} to compute the Symplectic Invariant are valid for any semi-positive Symplectic manifold 
(in particular the Symplectic Invariants are well defined). Hence, the Symplectic Invariant is well defined for a del-Pezzo surface. \\ 
\hf \hf The enumerative invariant $n^j_{2,\beta}$ is well defined because the complex structure on a del-Pezzo surface is regular as we show in section \ref{genus_two_regular}. 
That means that the expected dimension of the moduli space (which is $\delta_{\beta}-1$) is same as the actual dimension 
of the moduli space. Since passing through a point cuts the dimension by one, we conclude that $n^j_{2, \beta}$ is a well defined number. \\ 
\end{rem}

\begin{rem}
One of the things we need to know, in order to compute the correction term is the number of 
genus zero curves in $X$ passing through $\delta_{\beta}-1$ points that have a cusp. This 
number is computed in our paper \cite{IB_SDM_RM_VP}. This is defined by counting the 
number of zeros of the section an appropriate bundle on the moduli space of rational curves. 
We show that the section is transverse to the zero set (\cite[Lemma 6.2, Page 14]{IB_SDM_RM_VP}). 
Hence the characteristic number of rational cuspidal curves is well defined and enumerative. 
\end{rem}


\begin{rem}\label{automorphisms}
The symplectic invariant $\mathrm{RT}_{g, \beta}$ does not depend on $j$. Neither 
does the correction term $\mathrm{CR}_{g, \beta}$. It is only the enumerative 
invariant $n_{g, \beta}^j$ that depends on $j$. Equation \eqref{rt_equal_eg_plus_cr} 
implies that
$$
n^{j}_{g, \beta}\,=\, \frac{\mathrm{RT}_{g, \beta} - \mathrm{CR}_{g, 
\beta}}{|\textnormal{Aut}(\Sigma_g^j)|}\, .
$$ It is because of the presence of the 
factor $|\textnormal{Aut}(\Sigma_g^j)|$, that $n^j_{g, \beta}$ depends on $j$. If 
$g\,=\,2$ and $j$ is a generic complex structure on a genus two surface, then 
$|\textnormal{Aut}(\Sigma_g^j)|\,=\,2$.
\end{rem}


\section{Computation of the symplectic invariant}
\label{SI}

We compute the symplectic invariant $\mathrm{RT}_{2,\beta}$ using formulas 
$(1.1)$ and $(1.2)$ in \cite[page 263]{RT}. 
Throughout the discussion $X$ is a complex del-Pezzo surface 
(either $\mathbb{P}^2$ blown up at $k \leq 8$ points or $\mathbb{P}^1 \times \mathbb{P}^1$).\\
\hf \hf Let $e_0 := [\textnormal{pt}] \in H_0(X, \mathbb{Z})$ denote the class of a point 
in $X$ and let $e_{n+1} := [X] \in H_4(X, \mathbb{Z})$ denote the class of the whole space. 
Let $e_1\, , e_2\, , \cdots\, , e_n$ be a basis for $H_2(X,\,
\mathbb{Z})$. 
For notational conveniences, we will also use 
\[ f_1 \,:=\, e_1, ~f_2 \,:=\, e_2, ~ \cdots, ~ f_{n} \,:=\, e_n\, .\] 
The reason for this duplication of notation will be clear soon. 
Note that there is no such thing as $f_0$ or $f_{n+1}$. 
Hence, if we encounter a term $f_i$, it is understood that $i$ is between $1$ and 
$n$.

Define
$$
g_{ij}\,:=\, e_{i} \cdot e_j \ \ \qquad \textnormal{ and }\ \ \qquad g^{ij}\,:=
\,\Big(g^{-1}\Big)_{ij}\, .
$$
If the degree of $e_i$ and $e_j$ do not add up to the dimension of $X$, then define 
$g_{ij}\,:=\, 0$. Also, define the set of points $p_1, p_2, \ldots, p_{\delta_{\beta}-1}$ collectively as $\mu$, i.e. 
\[ \mu \,:=\, (p_1\, , p_2\, , \cdots\, , p_{\delta_{\beta}-1})\, . \]
We will also be using the Einstein summation convention, since it will 
make the subsequent computation much easier to read.

To compute the symplectic invariant, using
\cite[page 263, (1.2)]{RT}, it follows that 
\begin{align}
\mathrm{RT}_{2, \beta}(\emptyset; [\mu]) & =\mathrm{RT}_{1, \beta}(e_i, e_j; [\mu]) g^{ij} \nonumber \\ 
& = \mathrm{RT}_{0, \beta} (e_i, e_j, e_k, e_l; [\mu]) g^{ij} g^{kl} \nonumber \\
& = \mathrm{RT}_{0, \beta} (\textnormal{pt}, X, e_k, e_l; [\mu]) g^{kl} 
+ \mathrm{RT}_{0, \beta} (X, \textnormal{pt}, e_k, e_l; [\mu]) g^{kl} \nonumber \\ 
& ~~ + \mathrm{RT}_{0, \beta} (e_i, e_j, \textnormal{pt}, X; [\mu]) g^{ij}
+ \mathrm{RT}_{0, \beta} (e_i, e_j, X, \textnormal{pt}; [\mu]) g^{ij} \nonumber \\ 
& ~~ + \mathrm{RT}_{0, \beta} (f_i, f_j, f_k, f_l; [\mu]) g^{ij} g^{kl} \nonumber \\ 
& = 4 n_{0, \beta} (\beta \cdot f_i) (\beta \cdot f_j) g^{ij} 
+ \mathrm{RT}_{0, \beta} (f_i, f_j, f_k, f_l; [\mu]) g^{ij} g^{kl}.\label{rt_2_eqn_one}
\end{align}
Next, using 
\cite[page 263, (1.1)]{RT} it follows that
$$
\mathrm{RT}_{0, \beta} (f_i, f_j, f_k, f_l; [\mu]) g^{ij} g^{kl}
$$
$$
=\,
\mathrm{RT}_{0, \beta} (f_i, f_j, \textnormal{pt}; [\mu]) \mathrm{RT}_{0,0} (X, f_k, f_l; 
\emptyset) g^{ij} g^{kl}
+ \mathrm{RT}_{0, 0} (f_i, f_j, X; \emptyset) \mathrm{RT}_{0, \beta} 
(f_k, f_l, \textnormal{pt}; [\mu]) g^{ij}g^{kl}
$$
\begin{equation}\label{rt_eqn_two}
+\sum_{\beta_1+ \beta_2= \beta} \binom{\delta_{\beta}-1}{\delta_{\beta_1}} n_{0, \beta_1} n_{0, \beta_2}
(\beta_1 \cdot f_i) (\beta_1 \cdot f_j) (\beta_1 \cdot f_m) (\beta_2 \cdot f_n) (\beta_2 \cdot f_k) (\beta_2 \cdot f_l) 
g^{ij} g^{kl} g^{mn}\, .\\ 
\end{equation}
Note that 
\begin{align}
\mathrm{RT}_{0, 0} (X, f_k, f_l; \emptyset) g^{kl} & = (f_k \cdot f_l) g^{kl} = b_2(X), \label{betti_number}
\end{align}
where $b_2(X)$ is the dimension of $H_2(X, \mathbb{Z})$. 

Next, we observe that 
\begin{align}
(\gamma_1 \cdot f_i) (\gamma_2 \cdot f_j) g^{ij} & = \gamma_1 \cdot \gamma_2 \qquad \forall 
\gamma_1, ~\gamma_2 \in H_2(X, \mathbb{Z}). \label{beta_square}
\end{align}
Using equation \eqref{rt_2_eqn_one}, \eqref{rt_eqn_two}, \eqref{betti_number} and 
\eqref{beta_square}
it follows that 
\begin{equation}\label{rt_blow_up}
\mathrm{RT}_{2, \beta} (\emptyset; [\mu]) \,=\, (4 + 2 b_2(X)) n_{0, \beta} \beta \cdot \beta + 
\sum_{\beta_1+ \beta_2= \beta} \binom{\delta_{\beta}-1}{\delta_{\beta_1}} \beta_1 ^2 \beta_2 ^2 
(\beta_1 \cdot \beta_2) n_{0, \beta_1} n_{0, \beta_2}\, . 
\end{equation}

\section{Computation of the correction term: boundary contribution} 
\label{boundary_contribution}
\hf \hf We will be now compute the boundary contribution to the Sympelctic Invariant by closely following Zinger's thesis \cite{zinger_phd}.  
The results we will be using have been published in two papers, namely \cite{ZiSG} and \cite{g2p2and3}. However, we feel it will be 
easier and more convenient for the reader if we follow and refer to Zinger's thesis \cite{zinger_phd} because the entire material is 
presented in one single document. Furthermore, some of the arguments/exposition in the thesis \cite{zinger_phd} is slightly easier 
to follow and modify in the cases we need (when the target space is not $\mathbb{P}^2$, but a del-Pezzo surface). Henceforth, we will 
be referring to Zinger's thesis \cite{zinger_phd}. His thesis is available online at the following link 
\[ \textnormal{\url{http://dspace.mit.edu/handle/1721.1/8402?show=full}} \] 
\hf \hf Let us now explain how we will use the results of Zinger's thesis \cite{zinger_phd} to compute the 
boundary contribution. First of all, we will describe the boundary strata that contributes to the Symplectic Invariant.
Along the way, we will justify why there are no other boundary contributions. 
In order to do that, we will 
use a dimension counting argument that uses the fact that the complex structure on a del-Pezzo 
surface is genus two regular. \\ 
\hf \hf Next, we will express the boundary contribution from each of the 
strata by expressing them as an intersection number on the moduli space of rational curves. 
This step will be justified by modifying the arguments in the proofs of the relevant theorems in Zinger's thesis \cite{zinger_phd} 
(where he proves these assertions when the target space is $\mathbb{P}^2$).  
Finally, we will proceed to calculate these intersection numbers, 
by using the results of our earlier paper \cite{IB_SDM_RM_VP}. We will now proceed to describe the boundary
(the following Lemma and its proof is closely based on what is discussed in \cite[Section 7.1, Page 135]{zinger_phd}).  
\begin{lmm}
\label{bd_description_lm}
Let $(\Sigma, j)$ be a smooth genus two 
Riemann surface with a fixed complex structure $j$ and 
let $\beta \in H_2(X, \mathbb{Z})$ be a homology class.  
Given a perturbation $\nu$ and a positive real number $t$,  consider the moduli space 
$\mathcal{M}_{2, \delta_{\beta}-1}^{t \nu, j}(X, \beta; \mu; \emptyset)$.
Let 
\[(u_t, \overline{y}_t):= (u_t, y_1(t), y_2(t), \ldots, y_{\delta_{\beta}-1}(t)) \in  \mathcal{M}_{2, \delta_{\beta}-1}^{t \nu, j}(X, \beta; \mu; \emptyset),\]
where $\overline{y}_t$ is a collection of $\delta_{\beta}-1$ distinct points on the \textit{domain} (i.e. $\Sigma$). As $t$ goes to zero, 
the map $(u_t, \overline{y}_t)$ must converge to  a stable map $(u, \overline{y})$, where $(u, \overline{y})$  
can be only one of the following objects:  
\begin{enumerate}
\item \label{enum_inv_case}The map $(u, \overline{y})$ is a non multiply covered holomorphic map (of degree $\beta$) with 
a \textit{smooth} domain $\Sigma$.\\ 

\item \label{bd_inv_case} The stable map $(u, \overline{y})$ is a holomorphic \textit{bubble} 
map (of degree $\beta$) where the domain  is a tree of $S^2$ attached to the principal component  
$\Sigma$ (with the marked points distributed on the whole domain)
and the map $u$, restricted to $\Sigma$  is constant. 
\end{enumerate}
\end{lmm}

\begin{rem}
Note that the cardinality of the set 
appearing in Lemma \ref{bd_description_lm}, case \eqref{enum_inv_case} is our desired enumerative invariant. 
The objects arising in Lemma \ref{bd_description_lm}, case \eqref{bd_inv_case}
are the excess boundary contribution to the Symplectic Invariant.  
\end{rem}

\begin{rem}
Note that in case \eqref{bd_inv_case}, at least one $S^2$ is attached to the principal component $\Sigma$.  
\end{rem}

\noindent \textbf{Proof:} Note that when $(u_t, \overline{y}_t)$ converges to a either a bubble map (where at least one $S^2$ is attached to the 
$\Sigma$) or a multiply covered map, there are a priori, three possibilities, namely: 
\begin{enumerate} 
\label{bd_ct_full}
\item \label{bd_ct_case1} The map $u$, restricted to the principle component $\Sigma$ is simple (i.e. non multiply covered) and the tree 
contains at least one $S^2$. 
\item \label{bd_ct_case2} The map $u$, restricted to $\Sigma$ is multiply covered. 
\item \label{bd_ct_case3} The map $u$, restricted to $\Sigma$  is constant and the tree contains at least one $S^2$.
\end{enumerate}
To prove our claim, it suffices to show that 
cases \eqref{bd_ct_case1} and \eqref{bd_ct_case2} can not occur.\\ 
\hf \hf Let us first justify why case \eqref{bd_ct_case1} can not occur. 
Since $u$ restricted to $\Sigma$ is simple (non-multiply covered), it is not constant. 
Suppose the degree of $u$ restricted to $\Sigma$ is $\beta_1$ and the degree on the bubble component is 
$\beta_2$. Hence $\beta = \beta_1 + \beta_2$. Since $u$ restricted to $\Sigma$ is non constant, we conclude that 
$\beta_1 \neq 0$. Suppose there is exactly one $S^2$ attached to the principal component $\Sigma$. Then this 
object will pass through $\langle x_1, \beta_1 \rangle -2 + \langle x_1, \beta_2 \rangle-1$ generic points. But 
\[ \delta_{\beta}-1 = \langle x_1, \beta \rangle -2 >  \langle x_1, \beta_1 \rangle -2 + \langle x_1, \beta_2 \rangle-1.\] 
Hence, this object can not pass through $\delta_{\beta}-1$ generic points. A similar argument holds if there are more than one 
$S^2$ attached to the principal component $\Sigma$. \\ 
\hf \hf Next, let us justify why case \eqref{bd_ct_case2} can not occur. Let us first assume that there is no $S^2$ attached to the $\Sigma$, but 
$u$ is multiply covered. We will use a dimension counting argument by studying the degree of the underlying reduced 
curve to show that such a $u$ can not occur. 
Since $u$ is multiply covered, we can assume that $u := v \circ \phi$, where $\phi$ is a holmorphic map from $\Sigma$ to itself, with degree greater than $1$ 
and $v$ is a holomorphic cap from $\Sigma$ to $X$.
Let us assume the degree of $\phi$ is $r$. Assume that $v$ represents the class $\beta^{\prime}$. Hence 
\[ \beta = r \beta^{\prime}. \] 
We now note that if $r >1$, then 
\[ \langle x_1, \beta^{\prime} \rangle -1 < \langle x_1, \beta \rangle -1 = \delta_{\beta}-1. \] 
Hence, this curve will not be able to pass through $\delta_{\beta}-1$ generic points. The same argument holds 
if a certain number of $S^2$ were attached to the $\Sigma$ and $u$ was multiply covered restricted to $\Sigma$. This proves our claim. \qed \\

\subsection{Description of the boundary strata} 
\verb+ +\newline 
\hf \hf We will now describe the 
boundary strata, that arising from the objects 
in Lemma \ref{bd_description_lm}, case \eqref{bd_inv_case}. 
Before that let us recapitulate some notations about moduli space of rational curves.  \\
\hf \hf Consider 
the moduli space of rational degree $\beta$ curves on $X$ that 
represent the class $\beta\,\in\, H_2(X,\, {\mathbb Z})$ 
and are equipped with $n$ ordered marked points. 
Let $\mathcal{M}_{0,n}(X, \beta)$
denote the equivalence classes of such curves, i.e. if $\phi:\mathbb{P}^1\longrightarrow \mathbb{P}^1$ 
is an automorphism of $\mathbb{P}^1$, then 
\[ [u, y_1, \ldots, y_n] = [u \circ \phi, \phi^{-1}(y_1), \ldots, \phi^{-1}(y_n)]. \]
In other words,
$$
\mathcal{M}_{0,n}(X, \beta)\,:=\, \{ (u, y_1, \cdots, y_n) \,\in\,
\mathcal{C}^{\infty}(\mathbb{P}^1, X) \times
(\mathbb{P}^1)^n\,\mid ~ \overline{\partial}u\,=\,0, ~~u_*[\mathbb{P}^1]
\,= \,\beta \}/\text{PSL} (2, \mathbb{C})\, , 
$$
with $\text{PSL}(2, \mathbb{C})$ acting diagonally on $\mathbb{P}^1\times
(\mathbb{P}^1)^n$. 
For any $k\,\leq \,n$, let
$$\mathcal{M}_{0,n}(X, \beta; p_1, \cdots, p_{k})\,\subset\,\mathcal{M}_{0,n}(X, \beta)$$
be the subspace consisting of rational curves with
$n$ marked points such that the $i$-th marked point is $p_i$
for all $1\,\leq\, i\, \leq\, k$, so,
$$
\mathcal{M}_{0,n}(X, \beta; p_1, \cdots, p_{k})
\,:=\, \{ [u, y_1, \cdots, y_n] \,\in\, \mathcal{M}_{0,n}(X, \beta)\,\mid~ u(y_i) \,=\,
p_i~ \ \forall~ i \,= \,1, \cdots ,k\}\, .
$$
Let $\overline{\mathcal{M}}_{0,n}(X, \beta)$ denote the stable map 
compactification of $\mathcal{M}_{0,n}(X, \beta)$. \\
\hf \hf We will now concretely describe the space that arises in Lemma \ref{bd_description_lm}, case \eqref{bd_inv_case}. 
Suppose there is only \textit{one} $S^2$ attached to the principal component 
as described in the following picture. 
\begin{figure}[!htb]
  \begin{minipage}[b]{0.3\textwidth}
  \includegraphics[scale=0.4]{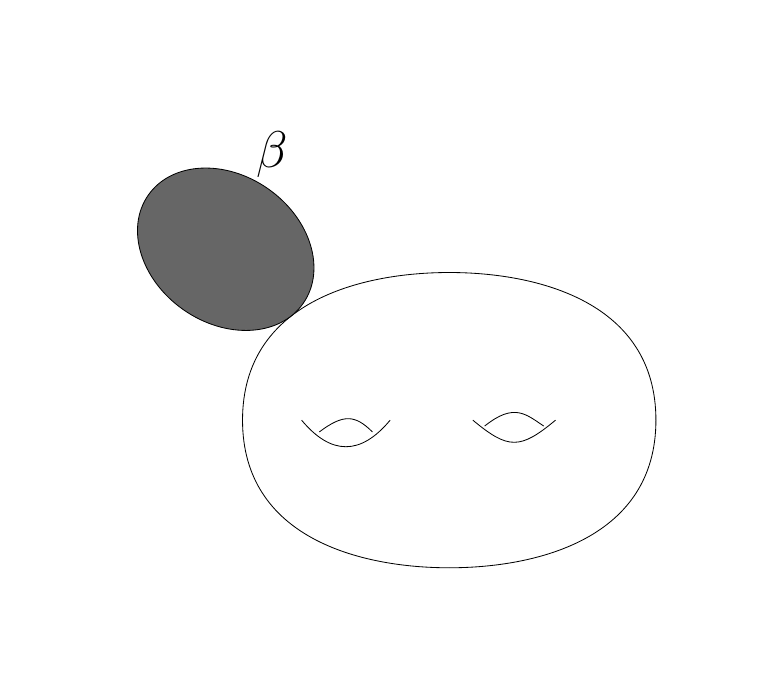}    
  \caption{$\mathcal{B}^1_1$, $\mathcal{B}^2_1$ and $\mathcal{B}^3_1$.}
  \end{minipage}%
\end{figure}
This space can be identified with the space $\Sigma \times \mathcal{M}_{0,\delta_{\beta}}(X, \beta; \mu)$. 
Note that there are $\delta_{\beta}$ marked points in $\mathcal{M}_{0,\delta_{\beta}}(X, \beta; \mu)$, 
not $\delta_{\beta}-1$. The first $\delta_{\beta}-1$ 
points are mapped to $\mu$ and the last point is a free marked point. This free marked 
point is identified with a point on $\Sigma$ and hence we can form a bubble map. Hence the space can be identified 
with $\Sigma \times \mathcal{M}_{0,\delta_{\beta}}(X, \beta; \mu)$.\footnote{We note 
that all the marked points will have to lie on the $S^2$ component; none of them can lie 
on the principal component, since the map is constant there and hence goes to the value of 
where the nodal point goes.}\\ 
\hf \hf Let us define $\mathcal{B}^1_1$ to be the subspace of $\Sigma \times \mathcal{M}_{0,\delta_{\beta}}(X, \beta; \mu)$ 
where the sphere does not have a cusp at the point attached to $\Sigma$, i.e. 
\begin{align}
\mathcal{B}^1_1 &:= \{ (x, [u, \overline{y}, y]) \in \Sigma \times \mathcal{M}_{0,\delta_{\beta}}(X, \beta; \mu): du|_{y} \neq 0\}.  \label{b11}
\end{align}
\hf \hf Next, given a genus two Riemann surface $(\Sigma, j)$, it contains six distinguished Weierstrass 
points, namely $z_1, z_2, \ldots, z_6 \in \Sigma$. Let 
\[ \Sigma^* := \Sigma - \{z_1, \ldots, z_6\}. \]
Let us define $\mathcal{B}^2_1$ to be the subspace of  
$\Sigma^* \times \mathcal{M}_{0,\delta_{\beta}}(X, \beta; \mu)$ where the curve $[u,y]$ has a 
cusp at the free marked point (i.e. $du|_y = 0$).
In other words 
\begin{align}
\mathcal{B}^2_1 &:= \{ (x, [u, \overline{y}, y]) \in \Sigma^* \times \mathcal{M}_{0,\delta_{\beta}}(X, \beta; \mu): du|_{y} = 0\}. \label{b21} 
\end{align}
\hf \hf Next, let us define $\mathcal{B}^3_1$ to be the subspace of $\{z_1, \ldots, z_6\} \times \mathcal{M}_{0,\delta_{\beta}}(X, \beta; \mu)$ 
where the curve $[u,y]$ has a 
cusp at the free marked point (i.e. $du|_y = 0$). In other words, 
\begin{align}
\mathcal{B}^3_1 &:= \{ (x, [u, \overline{y}, y] \in \{z_1, \ldots, z_6 \} \times \mathcal{M}_{0,\delta_{\beta}}(X, \beta; \mu): du|_{y} = 0\}.  \label{b31}
\end{align}
The significance of these Weierstrass 
points will be clear when we follow Zinger's thesis \cite{zinger_phd} to compute the boundary contribution. \\
\hf \hf We now make a simple observation about what can not occur in the boundary. 

\begin{lmm}
\label{bd3}
Consider a configuration where there is at least one $S^2$ attached to the principal component and 
another $S^2$ attached to this $S^2$. 
\begin{figure}[!htb]
  \begin{minipage}[b]{0.3\textwidth}
  \includegraphics[scale=0.4]{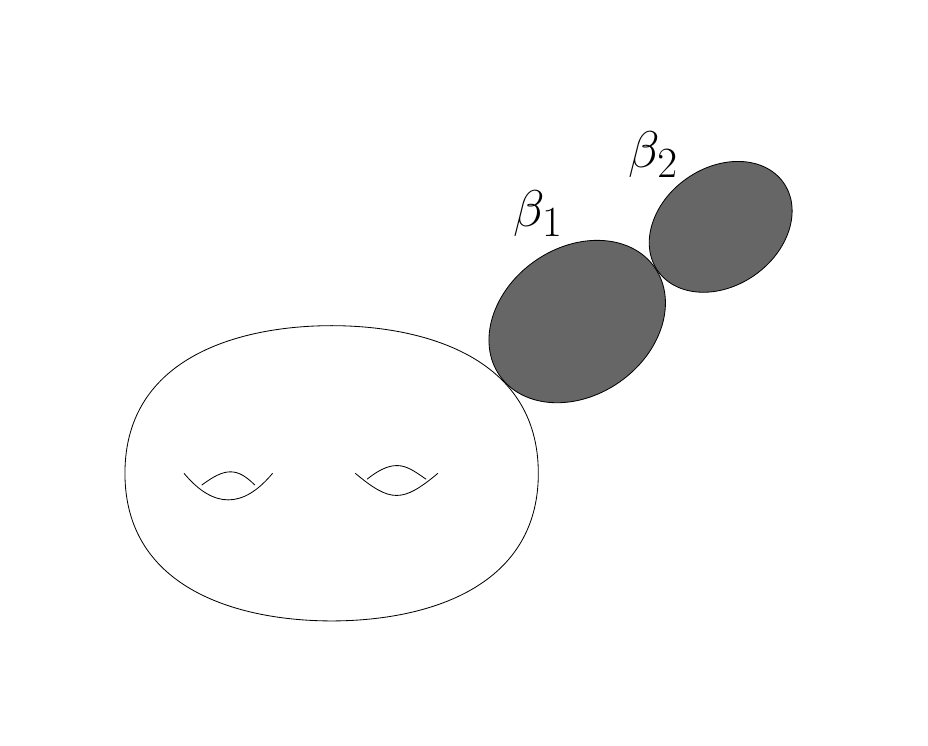}    
  \caption{Not Possible}
  \end{minipage}%
\end{figure}
This configuration can not occur in the boundary. 
\end{lmm}

\noindent \textbf{Proof:} This follows from a dimension counting argument. Suppose the degree on the two spheres are $\beta_1$ and $\beta_2$. Since  
\[ \langle x_1, \beta_1 \rangle -1 + \langle x_1, \beta_2 \rangle -1 = \langle x_1, \beta  \rangle -2 \]
we do not have any further degree of freedom to place the free marked point (the free marked point has to go to 
where ever the $\Sigma$ is mapped to). \qed \\  

\hf \hf Finally, let us describe the boundary strata where two $S^2$ are attached to the principal 
component as described in the following picture. 
\begin{figure}
[!ht]
  \begin{minipage}[b]{0.3\textwidth}
  \includegraphics[scale=0.4]{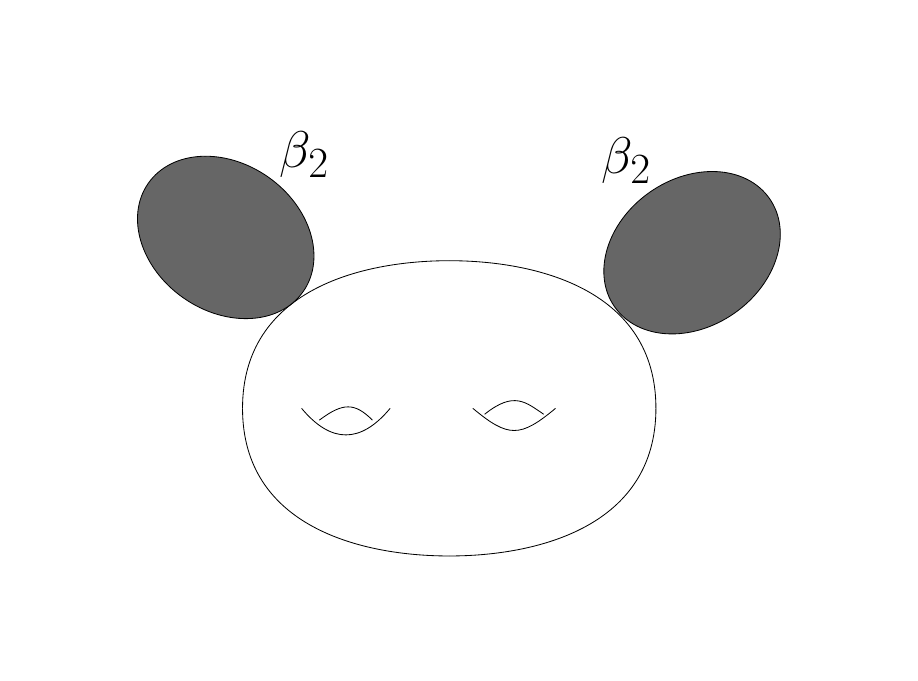}    
  \caption{The stratum $\mathcal{B}^1_2$.}
  \label{fig_b12}
  \end{minipage}%
\end{figure}
This space is identified with  
\[(\Sigma \times \Sigma-\Delta_{\Sigma}) \times  \mathcal{V}_2(\mu), \] 
where 
$\Delta_{\Sigma}$ is the diagonal inside $\Sigma \times \Sigma$ and 
$\mathcal{V}_2(\mu)$ is the space of 
two component degree $\beta$ rational curves with $\delta_{\beta}-1$ marked points (distributed on the domain, which is a wedge of two spheres) 
going to the constraints $\mu$. We will denote this space as $\mathcal{B}^1_2$.\\  \newpage 

\hf \hf Next, we make another observation about what can not occur in the boundary.

\begin{lmm}
\label{bd4}
Consider the following configuration, 
where three or more $S^2$ are attached to the 
principal component. 
\begin{figure}[!htb]
  \begin{minipage}[b]{0.3\textwidth}
  \includegraphics[scale=0.4]{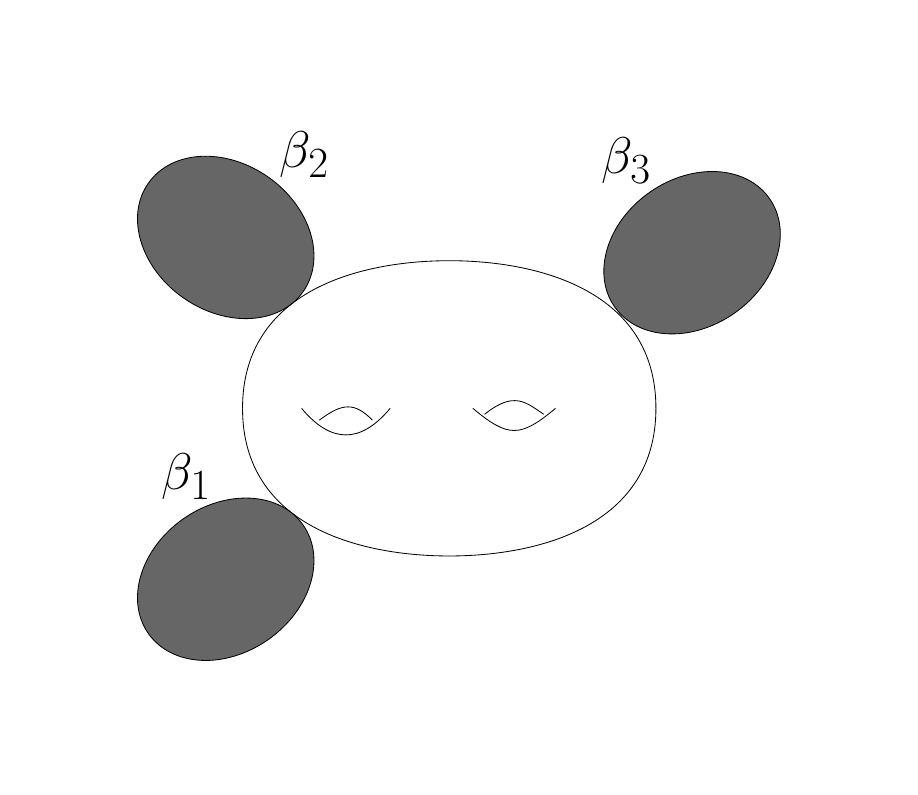}    
  \caption{Not Possible}
  \end{minipage}%
\end{figure}
This configuration can not occur in the boundary. 
\end{lmm}

\noindent \textbf{Proof:} Follows from a dimension counting argument, similar to the one given in the proof of Lemma \ref{bd3}. \qed \\ 

\hf \hf Hence, we have obtained a complete description of the boundary strata. 

\begin{lmm}
The total boundary strata $\mathcal{B}$ is the union of the spaces $\mathcal{B}^1_1$, 
$\mathcal{B}^2_1$, $\mathcal{B}^3_1$ and $\mathcal{B}^1_2$, i.e  
\begin{align}
\label{b_union_strata}
\mathcal{B} & = \mathcal{B}^1_1 \cup \mathcal{B}^2_1 \cup \mathcal{B}^3_1 \cup \mathcal{B}^1_2.    
\end{align}
\end{lmm}

\noindent \textbf{Proof:} Follows from the definitions of $\mathcal{B}^i_j$ 
and Lemma \ref{bd3} and \ref{bd4}. \qed 


\subsection{Calculation of the excess boundary contribution to the Symplectic Invariant} 
\verb+ +\\
\hf \hf We will now compute the boundary contribution by using (and modifying) the results (and proofs) of Zinger's thesis \cite{zinger_phd}. 
Let us first give a brief idea of what we will be doing. 
In his thesis \cite{zinger_phd}, Zinger computes the required 
boundary contribution when the target space $X:= \mathbb{P}^2$. The justification of 
the boundary contribution is as follows. In \cite[Chapter 8]{zinger_phd}, he expresses the 
relevant boundary contribution as the number of (small) zeros of an affine bundle 
map between relevant bundles. In \cite[Chapter 9]{zinger_phd} he computes this number, by using 
the topological formulas obtained in \cite[Chapter 5]{zinger_phd}. 
The justification of the relevant Theorems in \cite[Chapter 8]{zinger_phd} 
is as follows. He uses the results 
of \cite[Chapter 7]{zinger_phd} (Theorem 7.2 in particular). Furthermore, the proof of Theorem \cite[Theorem 7.2]{zinger_phd} 
follows from the results of 
\cite[Chapter 3]{zinger_phd} (Theorem 3.29 in particular).
We now note the following important fact. 
The theorems in \cite[Chapter 7 and 8]{zinger_phd} 
are applicable when the target manifold is $\mathbb{P}^2$. 
However, the theorems of \cite[Chapter 3]{zinger_phd} (Theorem 3.29 in particular) are 
not specific to $\mathbb{P}^2$; they are applicable to any Symplectic manifold with a regular almost complex structure.  
Hence, to compute and justify the boundary contribution when 
$X$ is a del-Pezzo surface, we claim that the relevant Theorems of \cite[Chapter 7 and 8]{zinger_phd} 
continue to hold 
when $\mathbb{P}^2$ is replaced by a del-Pezzo surface. The justification of this crucial assertion 
is that the proof of these Theorems follow from \cite[Theorem 3.29]{zinger_phd}, which holds for a general 
Symplectic 
manifold with a regular almost complex structure. We then go on to compute the boundary contribution 
by using the topological formula presented in \cite[Chapter 5]{zinger_phd} (analogous to what Zinger does in \cite[Chapter 9]{zinger_phd}). 
This involves knowing the various intersection numbers on the moduli space of rational curves (and hence we use the results of our earlier 
paper \cite{IB_SDM_RM_VP}). \\ 
\hf \hf Finally, we note that in order to be able to use \cite[Theorem 3.29]{zinger_phd}, we need to show 
that 
del-Pezzo surfaces admit a nice family of metrics 
that are constructed by Zinger in \cite[Lemma 6.1]{zinger_phd} for $\mathbb{P}^2$. 
This is needed so that we can use the results of \cite[Theorem 3.29]{zinger_phd} 
(see the discussion on \cite[Chapter 7, Page 136, Paragraph 3]{zinger_phd}). 
In section \ref{nice_metric_section} we explain why del-Pezzo surfaces admit a 
nice family of metrics (this is Lemma \ref{nice_metric}). \\
\hf \hf We will now get into the details of the procedure we have described above. 
Let us first recapitulate the definition of the quantity $s_{\Sigma}^m$ 
defined by Zinger in \cite[Section 7.3, Page 140]{zinger_phd}. 
It is a section of the following bundle 
\[ s_{\Sigma}^m \in \Gamma (\Sigma, T^*\Sigma^{\otimes m} \otimes \mathcal{H}^{0,1}_{\Sigma}). \]
Let $\mathcal{H}^{0,1}_{\Sigma} \longrightarrow \Sigma$ be the bundle of $(0,1)$ forms on $\Sigma$. 
Equip $\Sigma$ with a metric $g$
and let $\{\psi_j\}_{j=1}^N$ 
be an orthonormal basis for $\mathcal{H}^{0,1}_{\Sigma}$. Then, given a point 
$x \in \Sigma$ and  
a tangent vector $w \in T_x\Sigma$, we define 
\begin{align*}
\{s_{\Sigma}^m(x)\}(w)&:= \sum_{j=i}^{N} \overline{\{ D^m \psi_j(x) \}(w)} \psi_j(x),
\end{align*}
where $D$ denotes the covariant derivative. In general, $s^m_{\Sigma}$ depends 
on the metric $g$ if $m>1$; however, $s^1_{\Sigma}$ does not depend on the metric, but only depends on the Riemann surface $(\Sigma,j)$. 
By \cite[Page 246]{GH}, $s^1_{\Sigma}(x)$ does not vanish and hence spans a sub-bundle of $\mathcal{H}^{0,1}_{\Sigma}$; we denote this sub-bundle 
by $\mathcal{H}_{\Sigma}^+$ and its orthogonal complement by $\mathcal{H}_{\Sigma}^-$. 
We denote the projection of $s^{(m)}_{\Sigma}$ onto the $\mathcal{H}_{\Sigma}^+$ and $\mathcal{H}_{\Sigma}^-$ components as 
$s^{(m,+)}_{\Sigma}$ and $s^{(m, -)}_{\Sigma}$ respectively. We note that the bundles $\mathcal{H}^{0,1}_{\Sigma}$
$\mathcal{H}_{\Sigma}^+$ and $\mathcal{H}_{\Sigma}^-$ are all trivial. 
We are now ready to compute the boundary contributions. \\ 

\begin{lmm}
\label{lm1}
The contribution from $\mathcal{B}^1_1$ (which we denote as $n^1_1$) 
is given by 
\begin{align}
n^1_1 & = \frac{2(\hat{x}_1 \cdot \hat{x}_1)}{(\beta\cdot \hat{x}_1)} n_{0, \beta} \nonumber\\
& -\frac{1}{(\beta\cdot \hat{x}_1)}\sum_{
\substack{\beta_1+ \beta_2\,=\, \beta, \\ \beta_1, \beta_2 \neq 0} }
\binom{\delta_{\beta}-1}{\delta_{\beta_1}} n_{0, \beta_1} n_{0, \beta_2} (\beta_1 \cdot \beta_2)(\beta_1 \cdot \hat{x}_1)
(\beta_2 \cdot \hat{x}_1). \label{c1x1_revised_again}
\end{align}
\end{lmm}

\noindent \textbf{Proof:} We claim that the contribution from $\mathcal{B}_1^1$ is  
the same as the number of (small) zeros of the affine bundle map 
$v \longrightarrow \alpha^1_1(v) + \nu$, where $\nu$ is a generic perturbation and  $\alpha^1_1$ is following the linear bundle map  
\begin{align*}
\alpha^1_1 & \in \Gamma (\Sigma \times \overline{\mathcal{M}}_{0, \delta_{\beta}}(X, \beta; \mu); \textnormal{Hom}(T\Sigma \otimes \mathbb{L}; \mathcal{H}^{0,1}_{\Sigma} \otimes 
\textnormal{ev}^*TX)) \qquad \textnormal{given by} \\ 
\{\alpha^1_1(x, [u, \overline{y}; y])\} (w \otimes v) &:=  \{s_{\Sigma}^1 (x)\}(w) \otimes \{du|_{y}\}(v).
\end{align*}
Here $\mathbb{L} \longrightarrow \overline{\mathcal{M}}_{0, \delta_{\beta}}(X, \beta; \mu)$ is the universal tangent bundle 
at the last (free) marked point. \\
\hf \hf Let us now justify this claim. 
This claim is precisely \cite[Corollary 8.7]{zinger_phd} when $X:=\mathbb{P}^2$ (also see the explanation in \cite[section 8.9]{zinger_phd}). 
We claim that the statement of \cite[Corollary 8.7]{zinger_phd} 
is true even when $\mathbb{P}^2$ is replaced   
by $X$, a del-Pezzo surface. To see why, we unwind the proof of \cite[Corollary 8.7]{zinger_phd}. 
The proof follows from \cite[Theorem 7.3]{zinger_phd}
(which is for $\mathbb{P}^2$), which in turn follows from 
\cite[Theorem 3.29]{zinger_phd}
(which is for a general Symplectic manifold with a regular almost complex structure). 
Hence, using \cite[Theorem 3.29]{zinger_phd}, 
we conclude that the statement of \cite[Theorem 7.3]{zinger_phd} holds for any del-Pezzo surface 
and hence the 
statement of \cite[Corollary 8.7]{zinger_phd}  holds as well for a general del-Pezzo surface.\\
\hf \hf We now actually compute the number of zeros of the affine bundle map. 
We will closely follow the proof of \cite[Lemma 9.5]{zinger_phd}, where he computes this number for the special case when $X:= \mathbb{P}^2$. 
The proof uses the topological results of \cite[Chapter 5]{zinger_phd}; 
these are general statements about vector bundles and 
manifolds and are not specific to moduli spaces of curves. 
Hence, the target space plays no role while using the results of \cite[Chapter 5]{zinger_phd}).\\ 
\hf \hf By \cite[Corollary 5.17]{zinger_phd}, 
the number of (small) zeros of the affine bundle map 
\[ v \longrightarrow \alpha^1_1 (v) + \nu \]
is given by 
\begin{align}
n^1_1 &= 2 
\langle c_1(\yii^*)\textnormal{ev}^*(-x_1), ~[\overline{\mathcal{M}}_{0, \delta_{\beta}}(X, \beta; \mu)] \rangle 
+ (2 \hat{x}_1^2 - 2 x_2([X])) n_{0, \beta} - 2 \langle \textnormal{ev}^*x_1, ~[\mathcal{Z}] \rangle, \label{n11_revised}
\end{align}
where $\mathcal{Z}$ is the subset of $\Sigma \times \overline{\mathcal{M}}_{0, \delta_{\beta}}(X, \beta; \mu)$ where the 
map $\alpha^1_1$ fails to be injective. 
To see how we obtain \eqref{n11_revised}, we use \cite[Corollary 5.17]{zinger_phd}, and set 
\begin{align*}
L_{\Sigma} & := T\Sigma \longrightarrow \Sigma, \qquad 
V_{\Sigma} := \Sigma \times \mathcal{H}^{0,1} \longrightarrow \Sigma, \\
L_{\mathcal{M}} & := \yii \longrightarrow \overline{\mathcal{M}}_{0, \delta_{\beta}}(X, \beta; \mu), \qquad \textnormal{and} 
\qquad V_{\mathcal{M}} := 
\textnormal{ev}^*TX \longrightarrow \overline{\mathcal{M}}_{0, \delta_{\beta}}(X, \beta; \mu).  
\end{align*}
Here $\yii$ and $\textnormal{ev}$ denote the universal tangent bundle and the evaluation map at the last (free) marked point respectively. 
Note that $V_{\Sigma}:= \Sigma \times \mathcal{H}^{0,1} \longrightarrow \Sigma$ is a trivial bundle, hence 
$c_1(V_{\Sigma}) =0$.\\ 
\hf\hf Next, we claim that the first intersection number appearing in the right hand side of \eqref{n11_revised} is given by 
\begin{align}
\langle c_1(\yii^*)\textnormal{ev}^*(-x_1), ~[\overline{\mathcal{M}}_{0, \delta_{\beta}}(X, \beta; \mu)] \rangle &\,=\,
\frac{(\hat{x}_1 \cdot \hat{x}_1)}{(\beta\cdot \hat{x}_1)} n_{0, \beta} \nonumber\\
& -\frac{1}{2 (\beta\cdot \hat{x}_1)}\sum_{
\substack{\beta_1+ \beta_2\,=\, \beta, \\ \beta_1, \beta_2 \neq 0} }
\binom{\delta_{\beta}-1}{\delta_{\beta_1}} n_{0, \beta_1} n_{0, \beta_2} (\beta_1 \cdot \beta_2)(\beta_1 \cdot \hat{x}_1)
(\beta_2 \cdot \hat{x}_1). \label{c1x1_revised}
\end{align}
A self contained proof of this is given in section \ref{ITC_genus2} (it is also proved in our paper \cite{IB_SDM_RM_VP}). \\
\hf \hf It remains to compute the last term in the right hand side of \eqref{n11_revised}. We claim that this number is zero, i.e. 
\begin{align}
\langle \textnormal{ev}^*(x_1), ~[\mathcal{Z}] \rangle & =0. \label{z_zero}
\end{align}
To see why, we first take a closer look at what is $\mathcal{Z}$. This is the subset of  
$\Sigma \times \overline{\mathcal{M}}_{0, \delta_{\beta}}(X, \beta; \mu)$ where the 
map $\alpha^1_1$ fails to be injective (or equivalently in this case, where $\alpha^1_1$ vanishes). 
Since $s^1_{\Sigma}$ is nowhere vanishing (\cite[Page 246]{GH}), 
this implies that $du|_y$ has to vanish. Hence, 
we are looking at the space of rational curves passing through $\delta_{\beta}-1$ generic points and that have a cusp; this 
is a finite set of points (since we take the product with $\Sigma$, $\mathcal{Z}$ is a one dimensional space). Now 
\begin{align*}
\langle \textnormal{ev}^*(x_1), ~[\mathcal{Z}] \rangle 
\end{align*}
geometrically denotes the number of rational curves through $\delta_{\beta}-1$ points with a cusp lying on a cycle that is Poincar\'{e} dual to 
$x_1$; this is clearly zero (since requiring the cusp to lie on some cycle is an extra condition). Hence, the number is zero. 
Combining \eqref{c1x1_revised}, \eqref{z_zero} and plugging it in \eqref{n11_revised} gives us the value of $n^1_1$. \qed \\

\hf \hf Next, we will compute the contribution from $\mathcal{B}^2_1$. 

\begin{lmm}
\label{lm2}
The contribution from $\mathcal{B}^2_1$ (which we denote as $2 n^2_1$) 
is given by 
\begin{align}
2 n^2_1 & = 4\Big(x_2([X]) - \frac{\hat{x}_1\cdot \hat{x}_1}{\beta\cdot \hat{x}_1} \Big) n_{0, \beta} +
4\sum_{\substack{\beta_1+ \beta_2= \beta, \\ \beta_1, \beta_2 \neq 0}} \binom{\delta_{\beta}-1}{\delta_{\beta_1}}
n_{0, \beta_1} n_{0, \beta_2} (\beta_1 \cdot \beta_2) \Big(
\frac{(\beta_1 \cdot \hat{x}_1) (\beta_2 \cdot \hat{x}_1)}{2 (\beta \cdot \hat{x}_1)} -1 \Big). \label{c_beta_revised_again}
\end{align}
\end{lmm}

\begin{rem}
It will be clear shortly why we denoted the contribution as $2 n^2_1$ and not $n^2_1$. 
\end{rem}

\noindent \textbf{Proof:} Let us first define $\mathcal{S}_1(\mu)$ to be the space 
of rational curves with a cusp, i.e. 
\begin{align*}
\mathcal{S}_1(\mu) &:= \{ [u, \overline{y}; y] \in \mathcal{M}_{0,\delta_{\beta}}(X, \beta; \mu): du|_y =0\}. 
\end{align*}
Hence, $\mathcal{B}^2_1 := \Sigma^* \times \mathcal{S}_1(\mu)$, where $\Sigma^*$ is $\Sigma$ minus the  
six distinguished Weierstrass points. \\ 
\hf \hf We claim that the contribution from $\mathcal{B}_1^2$ is  
the same as twice the number of (small) zeros of the affine bundle map 
$v \longrightarrow \alpha^2_1(v) + \nu$, where $\nu$ is a generic perturbation and  $\alpha^2_1$ is following the linear bundle map  
\begin{align*}
\alpha^2_1 & \in \Gamma (\Sigma \times \overline{\mathcal{S}}_1(\mu); \textnormal{Hom}(T\Sigma^{\otimes 2} \otimes \mathbb{L}^{\otimes 2}; 
\mathcal{H}^{-}_{\Sigma} \otimes 
\textnormal{ev}^*TX)) \qquad \textnormal{given by} \\ 
\{\alpha^2_1(x, [u, \overline{y}; y])\} (w \otimes v) &:=  \{s_{\Sigma}^{(2, -)} (x)\}(w) \otimes \{d^2 u|_{y}\}(v).
\end{align*}
\hf \hf Let us now justify this claim. 
This claim is precisely \cite[Corollary 8.14]{zinger_phd} when $X:=\mathbb{P}^2$ (also see the explanation in \cite[section 8.9]{zinger_phd}). 
We claim that the statement of \cite[Corollary 8.14]{zinger_phd} 
is true even when $\mathbb{P}^2$ is replaced   
by $X$, a del-Pezzo surface. To see why, we unwind the proof of \cite[Corollary 8.14]{zinger_phd}. 
The proof follows \cite[Theorem 7.3]{zinger_phd}
(which is for $\mathbb{P}^2$), which in turn follows from 
\cite[Theorem 3.29]{zinger_phd}
(which is for a general Symplectic manifold with a regular almost complex structure). 
Hence, using \cite[Theorem 3.29]{zinger_phd}, 
we conclude that the statement of \cite[Theorem 7.3]{zinger_phd} holds for any del-Pezzo surface 
and hence the 
statement of \cite[Corollary 8.14]{zinger_phd}  holds as well for a general del-Pezzo surface.\\
\hf \hf We now actually compute the number of zeros of the affine bundle map. 
We will now closely follow the discussion in \cite[Section 9.2]{zinger_phd} 
and the proof of \cite[Lemma 9.4]{zinger_phd} in particular.\\ 
\hf \hf Let us first apply the construction described at the beginning of 
\cite[Section 9.2, Page 170]{zinger_phd} (just before \cite[Lemma 9.4]{zinger_phd} is stated).
The linear map $\alpha^2_1$ is injective everywhere except over the zero set of $s^{(2, -)}_{\Sigma}$. 
By \cite[Section 8.5]{zinger_phd}, this section vanishes transversely at $z_1, z_2, \ldots, z_6$ (the six Weierstrass points of $\Sigma$).
This induces a nowhere vanishing section 
\begin{align*}
\tilde{s}_{\Sigma} & \in \Gamma (\Sigma; \textnormal{Hom}(\tilde{T}\Sigma, \mathcal{H}^-_{\Sigma}))  \qquad \textnormal{where} \qquad 
\tilde{T}\Sigma := T\Sigma^{\otimes 2} \otimes \mathcal{O}(z_1) \otimes \mathcal{O}(z_2) \otimes \ldots \otimes \mathcal{O}(z_6).  
\end{align*}
Here $\mathcal{O}(z_m)$ is the holomorphic line bundle on $\Sigma$ corresponding to the divisor $z_m$. 
Let us now construct a new bundle map  $\tilde{\alpha}^2_1$ which is as follows:   
\begin{align*}
\tilde{\alpha}^2_1 & \in \Gamma (\Sigma \times \overline{\mathcal{S}}_1(\mu); \textnormal{Hom}(\tilde{T}\Sigma \otimes \mathbb{L}^{\otimes 2}; 
\mathcal{H}^{-}_{\Sigma} \otimes 
\textnormal{ev}^*TX)) \qquad \textnormal{given by} \\ 
\{\tilde{\alpha}^2_1(x, [u, \overline{y}; y])\} (w \otimes v) &:=  \{ \tilde{s}_{\Sigma}(x)\}(w) \otimes \{d^2 u|_{y}\}(v).
\end{align*}
Now we note that 
the number of small zeros of the affine map 
\[ v \longrightarrow \alpha^2_1(v) + \nu \]
is same as the number of small zeros of the affine map  
\[ v \longrightarrow \tilde{\alpha}^2_1(v) + \nu. \] 
This is because for a generic perturbation, the zeros will not lie over any of the six points $z_1, \ldots, z_6$.\\
\hf \hf We now note that the map $\tilde{\alpha}^2_1$ is injective everywhere. Hence, using \cite[Lemma 5.14]{zinger_phd}, 
we conclude that the number of zeros of this map is given by 
\begin{align}
n^2_1 & = \langle c_1(\mathcal{H}_{\Sigma}^-\otimes \textnormal{ev}^*TX) - c_1(\tilde{T}\Sigma \otimes \yii), [\Sigma \times \overline{\mathcal{S}}_1(\mu)] \rangle \nonumber \\ 
      & = 2 |\mathcal{S}_1(\mu)|. \label{n21_revised}
\end{align}
Note that $\mathcal{S}_1(\mu)$ is a finite set (the cardinality of this set is the number of rational degree $\beta$ curves in $X$ passing through $\delta_{\beta}-1$ points). 
Hence, restricted to $\mathcal{S}_1(\mu)$, the bundle $\textnormal{ev}^*TX$ is trivial. Furthermore, $\mathcal{H}_{\Sigma}^{-}$ is a trivial bundle. 
Hence $\mathcal{H}_{\Sigma}^-\otimes \textnormal{ev}^*TX$ is a trivial bundle on $\Sigma \times \overline{\mathcal{S}}_1(\mu)$ and 
all the Chern classes are zero.\\
\hf \hf It remains to compute the number $|\mathcal{S}_1(\mu)|$. This is computed in our paper \cite{IB_SDM_RM_VP}; the formula we obtain in that 
paper is 
\begin{align}
|\mathcal{S}_1(\mu)| & = \Big(x_2([X]) - \frac{\hat{x}_1\cdot \hat{x}_1}{\beta\cdot \hat{x}_1} \Big) n_{0, \beta} +
\sum_{\substack{\beta_1+ \beta_2= \beta, \\ \beta_1, \beta_2 \neq 0}} \binom{\delta_{\beta}-1}{\delta_{\beta_1}}
n_{0, \beta_1} n_{0, \beta_2} (\beta_1 \cdot \beta_2) \Big(
\frac{(\beta_1 \cdot \hat{x}_1) (\beta_2 \cdot \hat{x}_1)}{2 (\beta \cdot \hat{x}_1)} -1 \Big). \label{c_beta_revised}
\end{align} 
Equation \eqref{c_beta_revised} substituted in \eqref{n21_revised} gives us the value of $n^2_1$. The total contribution from $\mathcal{B}^2_1$ is twice this 
number, i.e. it is $2 n^2_1$. \qed \\

\begin{lmm}
\label{lm3}
The contribution from $\mathcal{B}^3_1$ (which we denote as $6 n^3_1$) 
is given by 
\begin{align}
 18 n^3_1 & = 18\Big(x_2([X]) - \frac{\hat{x}_1\cdot \hat{x}_1}{\beta\cdot \hat{x}_1} \Big) n_{0, \beta} +
18\sum_{\substack{\beta_1+ \beta_2= \beta, \\ \beta_1, \beta_2 \neq 0}} \binom{\delta_{\beta}-1}{\delta_{\beta_1}}
n_{0, \beta_1} n_{0, \beta_2} (\beta_1 \cdot \beta_2) \Big(
\frac{(\beta_1 \cdot \hat{x}_1) (\beta_2 \cdot \hat{x}_1)}{2 (\beta \cdot \hat{x}_1)} -1 \Big). 
\end{align}
\end{lmm}

\begin{rem}
It will be clear shortly why we denoted the contribution as $18 n^3_1$ and not $n^3_1$. 
\end{rem}

\noindent \textbf{Proof:} We claim that the contribution from $\mathcal{B}^3_1$ is  
the same as three times the number of (small) zeros of the affine bundle map 
$v \longrightarrow \alpha^3_1(v) + \nu$, where $\nu$ is a generic perturbation and  $\alpha^3_1$ is following the linear bundle map  
\begin{align*}
\alpha^3_1 \in \Gamma ((\Sigma-\Sigma^*) \times \overline{\mathcal{S}}_1(\mu)& ; \textnormal{Hom}(T\Sigma^{\otimes 3} \otimes (\mathbb{L}^{\otimes 2}\otimes \mathbb{L}^{\otimes 3}); 
\mathcal{H}^{-}_{\Sigma} \otimes 
\textnormal{ev}^*TX)) \qquad \textnormal{given by} \\ 
\{\alpha(x, [u, \overline{y}; y])\} (w \otimes v_1 \oplus w \otimes v_2) &:=  \{s_{\Sigma}^{(3, -)} (x)\}(w) \otimes \{d^2 u|_{y}\}(v_1) + 
\{s_{\Sigma}^{(3, -)} (x)\}(w) \otimes \{d^3 u|_{y}\}(v_2).
\end{align*}
\hf \hf Let us now justify this claim. 
This claim is precisely \cite[Corollary 8.18]{zinger_phd} when $X:=\mathbb{P}^2$ (also see the explanation in \cite[section 8.9]{zinger_phd}). 
As before, we claim that the statement of \cite[Corollary 8.18]{zinger_phd} 
is true even when $\mathbb{P}^2$ is replaced   
by $X$, a del-Pezzo surface. Again, this is because  
\cite[Theorem 3.29]{zinger_phd} is valid for a general Symplectic manifold 
with a regular almost complex structure. This implies that 
\cite[Theorem 7.3]{zinger_phd} is valid for any del-Pezzo surface and that in turn implies that 
the statement of \cite[Corollary 8.18]{zinger_phd}  holds as well for a general del-Pezzo surface.\\
\hf \hf We now compute this number the number of zeros of this affine bundle map $\alpha^3_1$. 
Let us denote this number to be $6 n^3_1$.\footnote{We are including the factor of six, since the cardinality of $\Sigma-\Sigma^*$ is six; 
$n^3_1$ is the number of zeros of the affine map $\alpha^3_1 + \nu$ restricted to 
$\{z_i\} \times \overline{\mathcal{S}}_1(\mu)$. Moreover, this is consistent with 
notation followed by Zinger in \cite{zinger_phd}.}
The set $(\Sigma-\Sigma^*) \times \overline{\mathcal{S}}_1(\mu)$ is a finite set. 
Hence, 
\begin{align}
6 n^3_1 & = |(\Sigma-\Sigma^*) \times \overline{\mathcal{S}}_1(\mu)| \nonumber \\ 
      & = 6 |\overline{\mathcal{S}}_1(\mu)| \label{n31_revised}.
\end{align}
Equation \eqref{c_beta_revised} substituted in \eqref{n31_revised} gives us the value of $6 n^3_1$. The total contribution from $\mathcal{B}^3_1$ is $3$ 
times this number, i.e. $18 n^3_1$. \qed \\ 
\hf \hf Finally, we will compute the boundary contribution from $\mathcal{B}^1_2$. 

\begin{lmm}
\label{lm4}
The contribution from $\mathcal{B}^1_2$ (which we denote as $n^1_2$) 
is given by 
\begin{align}
n^1_2 & =  \sum_{
\substack{\beta_1+ \beta_2= \beta, \\ \beta_1, \beta_2 \neq 0} }
2\binom{\delta_{\beta}-1}{\delta_{\beta_1}} n_{0, \beta_1} n_{0, \beta_2}
(\beta_1 \cdot \beta_2)
\end{align}
\end{lmm}

\noindent \textbf{Proof:} We claim that the contribution from $\mathcal{B}^1_2$ is  
the same as three times the number of (small) zeros of the affine bundle map 
\[v \longrightarrow \alpha^1_2(v) + \nu, \] 
where $\nu$ is a generic perturbation and  $\alpha^1_2$ is following the linear bundle map  
\begin{align*}
\alpha^1_2 \in \Gamma (\Sigma_1 \times \Sigma_2 \times \overline{\mathcal{V}}_2(\mu) ) ; 
\textnormal{Hom}(T\Sigma_1 \otimes \mathbb{L}_1 \oplus T \Sigma_2 \otimes \mathbb{L}_2 & ; 
\mathcal{H}^{0,1}_{\Sigma} \otimes 
\textnormal{ev}^*TX) \qquad \textnormal{given by} \\ 
\{\alpha^1_2(x_1, x_2, [u_1, y_1; u_2, y_2; \overline{y}])\} (w_1 \otimes v_1 \oplus w_2 \otimes v_2) &:=  \{s^{(1)}_{\Sigma_1}(x_1)\}(w_1) \{du_1|_{y_1}\}(v_1)
\\ 
& + \{s^{(1)}_{\Sigma_2}(x_2)\}(w_2) \{du_2|_{y_2}\}(v_2).
\end{align*}
Here $\mathcal{V}_2(\mu)$ is the space of two component rational degree $\beta$ 
passing through $\delta_{\beta}-1$ generic points and 
$u_1, y_1$ denotes the map on the first sphere and $y_1$ denotes the nodal point that is attached to the second sphere, while 
$u_2, y_2$ denotes the map on the second sphere and $y_2$ denotes the nodal point that is attached to the first sphere. 
Let us also explain what are $\yii_1$ and $\yii_2$. Let us consider the component of $\mathcal{V}_2(\mu)$ where the degree on the 
first component is $\beta_1$, passing through $p_1, \ldots p_{\delta_{\beta_1}}$ and the degree on the second component is $\beta_2$, 
passing through $\delta_{\beta_2}$. 
There is a projection map $\pi_1$ and $\pi_2$ 
onto $\mathcal{M}_{0,\delta_{\beta_1}+1}(X, \beta_1; p_1, \ldots, p_{\delta_{\beta_1}})$ and 
$\mathcal{M}_{0,\delta_{\beta_2}+1}(X, \beta_2; p_1, \ldots, p_{\delta_{\beta_2}})$ respectively.  
The role of the extra (free) marked point is played by the nodal point. On top of these two moduli 
spaces we have the universal tangent bundle at the last marked point. The bundle $\yii_1$ and $\yii_2$ 
is the pullback via these two projection maps. We note that $\mathcal{V}_2(\mu)$ is the disjoint union of 
all the spaces we described as we vary over $\beta_1$ and $\beta_2$ and vary over how we distribute the 
marked points over the domain. What we just describes is the restriction of the two line bundles restricted 
to each of these disjoint components.   \\
\hf \hf Let us now justify our claim about counting the (small) number of zeros of the affine map $\alpha^1_2 + \nu$. 
This claim is precisely justified in \cite[Equation 8.31, Page 167]{zinger_phd} when $X:=\mathbb{P}^2$. 
As before, we claim that this statement remains true when $X$ is any del-Pezzo surface. This is because 
the statement of \cite[Corollary 8.7]{zinger_phd} is true when $X$ is a del-Pezzo surface 
(\cite[Theorem 3.29]{zinger_phd} is valid for a general Symplectic manifold which in turn implies 
\cite[Theorem 7.3]{zinger_phd} is valid for any del-Pezzo surface and that in turn implies that 
\cite[Corollary 8.7]{zinger_phd} is true when $X$ is a del-Pezzo surface). 
Furthermore, \cite[Corollary 8.7]{zinger_phd} implies that the contribution from $\mathcal{B}^1_2$ is the number 
of zeros of the affine map $v \longrightarrow \alpha^1_2(v) + \nu$ (by choosing the appropriate bubble type for $\mathcal{T}$).\\
\hf \hf Let us now compute this number. We note that the map $\alpha$ is injective everywhere 
(this is because $s^{(1)}_{\Sigma_1}$ and $s^{(1)}_{\Sigma_2}$ are nowhere vanishing and each of the component curves $u_1$ and $u_2$ 
do not have a cusp anywhere and intersect transversally since the points $\mu$ are in general position; in particular $du_1|_{y_1}$ and $du_2|_{y_2}$ are both non zero 
and span $TX|_{u_1(y_1)}$). Hence, by \cite[Lemma 5.14]{zinger_phd}, we conclude that the number of small zeros of the affine map $v \longrightarrow \alpha^1_2(v) + \nu$ 
is given by 
\begin{align}
n^1_2 &=  \langle c(\mathcal{H}^{0,1}_{\Sigma} \otimes \textnormal{ev}^*TX)c(T\Sigma_1 \otimes \mathbb{L}_1 \oplus T \Sigma_2 \otimes \mathbb{L}_2)^{-1}, 
~[\Sigma_1 \times \Sigma_2 \times \overline{\mathcal{V}}_2(\mu)] \rangle \nonumber \\
      & = \chi(\Sigma_1) \chi(\Sigma_2) |\overline{\mathcal{V}}_2(\mu)| \nonumber \\ 
      & = 4 |\overline{\mathcal{V}}_2(\mu)|. \label{n12_revised}
\end{align}
Note that this is basically the proof of \cite[Lemma 9.2]{zinger_phd} when $X:= \mathbb{P}^2$. 
While computing the Chern classes, we note that  
$\mathcal{H}^{0,1}_{\Sigma} \otimes \textnormal{ev}^*TX$ restricted to $\Sigma_1 \times \Sigma_2 \times \overline{\mathcal{V}}_2(\mu)$ 
is the trivial bundle (since $\overline{\mathcal{V}}_2(\mu)$ is a finite set) 
and hence all the Chern classes of that bundle are zero.  \\
\hf \hf It remains to compute $|\overline{\mathcal{V}}_2(\mu)|$. 
Since $|\mathcal{V}_2(\mu)|$ is the total number of two component rational degree $\beta$ 
passing through $\delta_{\beta}-1$ generic points and keeping track of the point at which the 
two components intersect, we conclude that 
\begin{align}
|\mathcal{V}_2(\mu) | \,=\, \frac{1}{2}\sum_{
\substack{\beta_1+ \beta_2= \beta, \\ \beta_1, \beta_2 \neq 0} }
\binom{\delta_{\beta}-1}{\delta_{\beta_1}} n_{0, \beta_1} n_{0, \beta_2}
(\beta_1 \cdot \beta_2)\, . \label{v2_number}
\end{align}
The factor of $\beta_1 \cdot \beta_2$ is there since we are keeping track of the point of 
intersection of the $\beta_1$ curve and the $\beta_2$ curve. Using \eqref{v2_number} in \eqref{n12_revised}, we get the value of $n^1_2$. \qed \\ 

\subsection{Proof of the Main Theorem} 

\verb+ + \\ 

\hf \hf Using Lemma \ref{lm1}, \ref{lm2}, \ref{lm3} and \ref{lm4} and equation \eqref{b_union_strata}, 
we conclude that the   
total correction term is given by  
\begin{align}
\mathrm{CR}_{2,\beta} &= n^1_1 + 2 n^2_1 + 18 n^3_1 + n^1_2.  \label{cr_revised}
\end{align}
Equation \eqref{rt_blow_up} gives us the value of $\mathrm{RT}_{2,\beta}$. 
Using the values of $\mathrm{RT}_{2,\beta}$ and $\mathrm{CR}_{2,\beta}$ and plugging it in equation \eqref{rt_equal_eg_plus_cr} 
gives us the enumerative invariant $n^j_{2, \beta}$. That gives us the formula obtained in Theorem \ref{main_thm}, i.e equation \eqref{main_formula_blwo_up}. \qed

\section{Intersection of Tautological Classes}
\label{ITC_genus2}
We now give a self contained proof of computing the intersection number that is computed in equation \eqref{c1x1_revised}. 
However, this is also proved in our paper \cite{IB_SDM_RM_VP}.

\begin{lmm}
\label{c1_divisor_ionel}
On $\overline{\mathcal{M}}_{0,\delta_{\beta}}(X, \beta; \mu)$, the following equality of divisors holds:
\begin{align}
c_1(\yii^*) &\,=\, \frac{1}{(\beta \cdot \hat{x}_1)^2}
\Big( (\hat{x}_1 \cdot \hat{x}_1) \mathcal{H} -2 (\beta \cdot \hat{x}_1) \textnormal{ev}_{\delta_{\beta}}^*(\hat{x}_1) +
\sum_{\substack{\beta_1+ \beta_2= \beta}}
\mathcal{B}_{\beta_1, \beta_2} (\beta_2 \cdot \hat{x}_1)^2 \Big),\label{chern_class_divisor}
\end{align}
where $\mathcal{H}$ is the locus satisfying the extra condition that the curve
passes through a given point, $\mathcal{B}_{\beta_1, \beta_2}$ denotes the
boundary stratum corresponding to the splitting into a
degree $\beta_1$ curve and degree $\beta_2$ curve with the last marked point
lying on the degree $\beta_1$ component. 
\end{lmm}

\begin{proof}
Let us first abbreviate and denote $\overline{\mathcal{M}}_{0,\delta_{\beta}}(X, \beta; \mu)$ 
by $\overline{\mathcal{M}}$. The proof we present is similar to the one given in \cite{ionel_genus1} for $\mathbb{P}^2$ 
(we also prove this in our paper \cite{IB_SDM_RM_VP}). Let
$\mu_1\, , \mu_2 \,\in\, X$ be two generic cycles in $X$ that represent the class
$\hat{x}_1$. Let $\widetilde{\mathcal{M}}$ be a cover of $\overline{\mathcal{M}}$ with two additional
marked points with the last two marked points lying on $\mu_1$ and $\mu_2$
respectively. More precisely,
\begin{align*}
\widetilde{\mathcal{M}} &\,:=\,\textnormal{ev}_{\delta_{\beta} +1}^{-1}(\mu_1) \cap
\textnormal{ev}_{\delta_{\beta} +2}^{-1}(\mu_2) \subset
\overline{\mathcal{M}}_{0, \delta_{\beta}+2}(X, \beta; \mu)\, .
\end{align*}
Here $\textnormal{ev}_i$ denotes the evaluation map at the $i^{\textnormal{th}}$ marked point. 
Note that the projection 
\[\pi\,:\, \widetilde{\mathcal{M}} \,\longrightarrow\,
\overline{\mathcal{M}}\] 
that forgets the last two marked points is a
$(\beta \cdot x_1)^2$--to--one map.

We now construct a meromorphic section
$$\phi\, :\,\widetilde{\mathcal{M}} \,\longrightarrow\, \yii_{\delta_{\beta}}^*$$
given by
\begin{align}
\phi ([u; y_{_1}, \cdots, y_{_{\delta_{\beta}-1}}; y_{_{\delta_{\beta}}};
y_{_{\delta_{\beta} +1}},
y_{_{\delta_{\beta}+2}}]) &:= \frac{(y_{_{\delta_{\beta}+1}} - y_{_{\delta_{\beta}+2}})
d y_{_{\delta_{\beta}}}}
{(y_{_{\delta_{\beta}}}-y_{_{\delta_{\beta}+1}})(y_{_{\delta_{\beta}}}-
y_{_{\delta_{\beta}+2}})}. \label{section_ionel}
\end{align}
Here $\yii_{\delta_{\beta}}$ is the universal tangent bundle over $\widetilde{\mathcal{M}}$ at the 
$\delta_{\beta}^{\textnormal{th}}$ marked point. \\
\hf \hf The right--hand side of \eqref{section_ionel} involves an abuse of notation: it is to be
interpreted in an affine coordinate chart and then extended as a meromorphic section
on the whole of $\mathbb{P}^1$. Note that on $({\mathbb P}^1)^3$, the holomorphic
line bundle
$$
\eta\, :=\,
q^*_1K_{{\mathbb P}^1}\otimes{\mathcal O}_{({\mathbb P}^1)^3}(\Delta_{12}
+\Delta_{13}-\Delta_{23})$$ is trivial, where $q_1\, :\, ({\mathbb P}^1)^3\,
\longrightarrow\,{\mathbb P}^1$ is the projection to the first factor and
$\Delta_{jk}\,\subset\, ({\mathbb P}^1)^3$ is the divisor consisting of all points
$(z_i\, ,z_2\, ,z_3)$ such that $z_j\,=\, z_k$. The diagonal action of ${\rm PSL}(2,
{\mathbb C})$ on $({\mathbb P}^1)^3$ lifts to $\eta$ preserving its
trivialization. The section $\phi$ in \eqref{section_ionel} is given by this trivialization of $\eta$.

Since $c_1(\yii_{\delta_{\beta}}^*)$ is the zero divisor minus the
pole divisor of $\phi$, we gather that
\begin{align*}
c_1(\yii_{\delta_{\beta}}^*) &\,=\, \{ y_{_{\delta_{\beta}+1}} = y_{_{\delta_{\beta}+2}}\}
-\{y_{_{\delta_{\beta}}}=y_{_{\delta_{\beta}+1}} \}-
\{y_{_{\delta_{\beta}}}= y_{_{\delta_{\beta}+2}}\}\, .
\end{align*}
When projected down to $\overline{\mathcal{M}}$, the divisor
$\{ y_{_{\delta_{\beta}+1}} = y_{_{\delta_{\beta}+2}}\}$ becomes
$(x_1\cdot x_1)\mathcal{H} + (\beta_2 \cdot x_1)^2 \mathcal{B}_{\beta_1, \beta_2}$,
while both the divisors $\{y_{_{\delta_{\beta}}}=y_{_{\delta_{\beta}+1}} \}$ and
$\{y_{_{\delta_{\beta}}}=y_{_{\delta_{\beta}+2}} \}$ become
$(\beta \cdot x_1) \textnormal{ev}_{\delta_{\beta}}^*(x_1)$.
Since $\widetilde{\mathcal{M}}$
is a $(\beta \cdot x_1)^2$--to--one cover of $\overline{\mathcal{M}}$, we obtain \eqref{chern_class_divisor}.
\end{proof}
Let us now compute the intersection number in equation \eqref{c1x1_revised}. 
First we note that 
\begin{align}
\langle \textnormal{ev}^*(\hat{x}_1) \mathcal{H}, ~[\overline{\mathcal{M}}_{0,\delta_{\beta}}(X, \beta; \mu)] \rangle &=
(\beta \cdot \hat{x}_1)n_{0,\beta}. \label{x1H}
\end{align}
To see why we note that the left hand side of \eqref{x1H} geometrically denotes the number of rational degree $\beta$ curves passing 
through $\delta_{\beta}$ points (the extra point comes by intersecting it with $\mathcal{H}$) and one marked point that 
lies on a cycle representing $\hat{x}_1$. Hence there is a factor of $\beta \cdot \hat{x}_1$ multiplied with $n_{0, \beta}$ 
since there are $\beta \cdot \hat{x}_1$ choices for the last marked point to go to after we have 
chosen a rational curve. \\
\hf \hf Next, we note that 
\begin{align}
\langle \textnormal{ev}^*(\hat{x}_1)^2,
~[\overline{\mathcal{M}}_{0,\delta_{\beta}}(X, \beta; \mu)] \rangle & = (\hat{x}_1 \cdot \hat{x}_1) n_{0,\beta}.  \label{x1x1}
\end{align}
To see why we note that the left hand side of \eqref{x1x1} denotes the number of rational degree $\beta$ curves in 
$X$ passing through $\delta_{\beta}-1$ generic points and any one more point that lies in the intersection of 
two cycles representing the class $\hat{x}_1$. Hence we have a factor of $\hat{x}_1 \cdot \hat{x}_1$ 
since there are $\hat{x}_1 \cdot \hat{x}_1$ choices for choosing the last point through which the curve will pass. \\ 
\hf \hf Finally, we note that 
\begin{align}
\langle \textnormal{ev}^*(x_1) \mathcal{B}_{\beta_1, \beta_2}, ~[\overline{\mathcal{M}}_{0,\delta_{\beta}}(X, \beta; \mu)] \rangle &=
\sum_{\substack{\beta_1+ \beta_2= \beta}}
\binom{\delta_{\beta}-1}{\delta_{\beta_1}} n_{0,\beta_1} n_{0,\beta_2} (\beta_1 \cdot \beta_2)(\beta_1 \cdot \hat{x}_1)
\nonumber \\
& = \frac{1}{2}\sum_{\substack{\beta_1+ \beta_2= \beta}}
\binom{\delta_{\beta}-1}{\delta_{\beta_1}} n_{0,\beta_1} n_{0,\beta_2} (\beta_1 \cdot \beta_2)(\beta_1 \cdot \hat{x}_1 + \beta_2 \cdot \hat{x}_1) \nonumber  \\
& = \frac{1}{2}\sum_{\substack{\beta_1+ \beta_2= \beta}}
\binom{\delta_{\beta}-1}{\delta_{\beta_1}} n_{0,\beta_1} n_{0,\beta_2} (\beta_1 \cdot \beta_2) (\beta \cdot \hat{x}_1).
\label{itc_easy}
\end{align}
To see why, we note that the left hand side of the first line of \eqref{itc_easy} denotes the number of ways we can place two component curves (of degrees $\beta_1$ and $\beta_2$ on the 
first and second component) meeting at common point 
going through a total of $\delta_{\beta}-1$ points and a marked point on the first component that goes to a cycle representing the class $\hat{x}_1$. Hence we 
have the factor of $\beta_1 \cdot \hat{x}_1$. The factor of $\beta_1 \cdot \beta_2$ is there because these are the choices we can make for the 
nodal point (the point at which the two spheres are attached).\\ 
\hf \hf Finally we note that equations \eqref{x1H}, \eqref{x1x1}, \eqref{itc_easy} combined with \eqref{chern_class_divisor} 
gives us \eqref{c1x1_revised}.

\section{Existence of a nice metric} 
\label{nice_metric_section}

In this is section we explain why del-Pezzo surfaces admit a nice family of metrics 
that are constructed by Zinger in \cite[Lemma 6.1]{zinger_phd} for $\mathbb{P}^n$. 
This is needed so that we can use the results of \cite[Theorem 3.29]{zinger_phd} 
(see the discussion on \cite[Chapter 7, Page 136, Paragraph 3]{zinger_phd}). 

\begin{lmm}
\label{nice_metric}
Let $X$ be an algebraic surface of dimension $m$. Then 
there exist $r_{X}>0$ and a smooth family of K\"ahler metrics $\{g_{X, q}: q \in X \}$ 
on $X$ with the following property. If $B_q(q^{\prime}, r) \subset X$ denotes the $g_{X,q}$ 
geodesic ball about $q^{\prime}$ of radius $r$, the triple $(B_q(q, r_X), J, g_{X,q})$ is isomorphic 
to a ball in $\mathbb{C}^m$ for all $q \in X$. 
\end{lmm}


\begin{proof}
This statement is true, when $X:= \mathbb{P}^n$; this is proved by Zinger in \cite[Lemma 6.1]{zinger_phd}. 
The present Lemma now follows, by embedding $X$ in some $\mathbb{P}^n$ and restricting the K\"ahler metric to $\mathbb{P}^n$. 
\end{proof}

\begin{rem}
Since del-Pezzo surfaces are algebraic, Lemma \ref{nice_metric} shows that they admit these families of metrics.  
\end{rem}

\section{Regularity of the complex structure} 
\label{genus_two_regular}
We now show that the complex structure on a del-Pezzo surface is genus two regular. 
Let us first recall the fact that the complex structure on $\mathbb{P}^n$ is genus two regular. 

\begin{lmm}
\label{p2_genus2_regular}
Let 
$(\Sigma_2, j)$ a compact genus $2$ Riemann surface with a 
complex structure $j$. Let 
$u\,:\, \Sigma_2 \,\longrightarrow\, \mathbb{P}^n$ be a degree $d$ holomorphic map. 
Then 
$$
H^1(\Sigma_2,\, u^*T\mathbb{P}^n) \,=\,0
$$
provided $d$ is greater than $2$. 
\end{lmm}
\begin{proof}
This is proved by Zinger in \cite[Corollary C.5, Page 236]{zinger_phd}.  
\end{proof}
We will now use the above theorem to obtain regularity of the complex structure on del-Pezzo surface.


\begin{lmm}
\label{dp_regular_g2}
Let $X$ be $\mathbb{P}^2$ blown up at $k$ points and 
$(\Sigma_2, j)$ a compact genus $2$ Riemann surface with a 
complex structure $j$. Let 
$u\,:\, \Sigma_2 \,\longrightarrow\, X$ be a holomorphic map representing the 
class $\beta \,:=\, d L - m_1 E_1- \ldots -m_k E_k \in H_2(X, \,\mathbb{Z})$, 
where $L$ and $E_i$ denote the class of a line and the exceptional divisors 
respectively. Then
$$
H^1(\Sigma_2,\, u^*TX) \,=\,0
$$
provided the following two conditions hold
\begin{itemize}
\item $u(\Sigma_2)$ intersects each of the $k$ exceptional divisors transversally, and

\item $d+\chi(\Sigma_2)\, >\, 0$.
\end{itemize}
\end{lmm}

\begin{proof}
Assume that $u$ satisfies the above two conditions.
Let $p\, :\, X\, \longrightarrow\, \mathbb{P}^2$ be the blow-up morphism. The
composition $p\circ u$ will be denoted by $\widehat{u}$. Let
$$
D\:=\, u^{-1} (\cup_{i=1}^k E_i)\, \subset\, \Sigma_2
$$
be the divisor on $\Sigma_2$ given by the intersection of $u(\Sigma_2)$ with the
exceptional divisor. We note that $D$ is reduced because of the first of the
two condition.

Let $dp\, :\, TX\, \longrightarrow\, p^* T\mathbb{P}^2$ be the differential of the
map $p$. This produces a short exact sequence on $\Sigma_2$
$$
0\, \longrightarrow\, u^*TX\, \stackrel{dp}{\longrightarrow}\, u^*p^* T\mathbb{P}^2\,=\,
\widehat{u}^*T\mathbb{P}^2\, \longrightarrow\, {\mathcal T}\, \longrightarrow\, 0\, ,
$$
where ${\mathcal T}$ is a torsion sheaf supported on $D$ and the dimension of its stalk
at each point of $D$ is one. The vector bundle $T\mathbb{P}^2$ is generated by
its global sections and hence this holds also for $\widehat{u}^*T\mathbb{P}^2$. So we
have
$$
H^1(\Sigma_2,\, u^*TX) \,=\,H^1(\Sigma_2,\, \widehat{u}^*T\mathbb{P}^2)\, .
$$
In view of the second condition, from Lemma \ref{p2_genus2_regular} 
we know
that $H^1(\Sigma_2,\, \widehat{u}^*T\mathbb{P}^2)\,=\, 0$. Hence
$H^1(\Sigma_2,\, u^*TX) \,=\,0$.
\end{proof}

\begin{rem}
The second condition on $u$ in Lemma \ref{dp_regular_g2} is automatically satisfied in our case since the 
curve $u$ passes through $\delta_{\beta}-1$ generic points. 
\end{rem}

The case of $\mathbb{P}^1 \times \mathbb{P}^1$ follows immediately from Lemma \ref{p2_genus2_regular}. In particular we obtain the following Lemma: 

\begin{lmm}
\label{p1xp1_g2_regular}
Let $X$ be $\mathbb{P}^1 \times \mathbb{P}^1$ 
and 
$(\Sigma_2, j)$ a compact genus $2$ Riemann surface with a 
complex structure $j$. Let 
$u\,:\, \Sigma_2 \,\longrightarrow\, X$ be a holomorphic map representing the 
class $\beta := a e_1 + b e_2 \in H_2(X, \,\mathbb{Z})$, 
where $e_1$ and $e_2$ denote the class of $[\mathbb{P}^1] \times [\textnormal{pt}]$ and 
$[\textnormal{pt}] \times [\mathbb{P}^1]$ 
in $H_2(\mathbb{P}^1 \times \mathbb{P}^1; \mathbb{Z})$ respectively. 
Then 
$$
H^1(\Sigma_2,\, u^*TX) \,=\,0
$$
provided $a$ and $b$ are both greater than $2$. 
\end{lmm}

\begin{proof} Follows immediately from Lemma \ref{p2_genus2_regular}.  
\end{proof}

\begin{rem}
Lemmas \ref{dp_regular_g2} and \ref{p1xp1_g2_regular} imply in particular, that the complex structure on the del-Pezzo surfaces 
are genus $2$ regular.   
\end{rem}

\section*{Acknowledgements}
The second author is grateful to Hannah Markwig and Yoav Len for fruitful discussions on this subject and informing us about what is known 
in this subject from the point of view of tropical geometry. He is also grateful to TIFR (Mumbai) and ICTS (Bangalore) for their hospitality where 
a major part of the project took place. The third author would like to thank NISER for its hospitality where a major part of the project took place.  
We are also grateful to Vamsi Pingali for writing the C++ program. The first author is partially supported by a J. C. Bose Fellowship. 
We are also very grateful to the referee for giving us helpful comments on the earlier version of this paper


\end{document}